\documentclass[11pt,a4paper,twoside]{amsart}

\usepackage[hmargin=2.5cm,bmargin=2.5cm,tmargin=3cm]{geometry}
\usepackage{soul}
\usepackage[normalem]{ulem}

\usepackage{hyperref}
\usepackage[english]{babel}
\usepackage{enumerate,colonequals,expdlist}
\usepackage{amssymb,amsmath,amsfonts,amsthm}
\usepackage{color,graphicx}
\usepackage{tikz,float}
\usetikzlibrary{patterns}
\usepackage[textsize=tiny]{todonotes} 
\usepackage{mathabx}
\usepackage[noadjust]{cite}

\theoremstyle{plain}
\newtheorem{thm}{Theorem}[section]
\newtheorem{prop}[thm]{Proposition}
\newtheorem{assum}[thm]{Assumption}
\newtheorem{rem}[thm]{Remark}
\newtheorem{cor}[thm]{Corollary}
\newtheorem{lemma}[thm]{Lemma}
\theoremstyle{definition}
\newtheorem{defin}[thm]{Definition}
\newtheorem{exa}[thm]{Example}
\numberwithin{equation}{section}
\numberwithin{figure}{section}

\newcommand{\abs}[1]{\lvert{#1}\rvert} 
\newcommand{\normsymb}{\|}
\newcommand{\norm}[2]{\normsymb{#1}\normsymb_{#2}} 

\newcommand{\RR}{\mathbb{R}} 
\newcommand{\CC}{\mathbb{C}} 
\newcommand{\NN}{\mathbb{N}} 

\newcommand{\e}{\mathrm{e}}

\newcommand{\cD}{{\mathcal D}}
\newcommand{\cE}{{\mathcal E}}
\newcommand{\cH}{{\mathcal H}}

 \def\Graph{\mathcal{G}}

 \def\mV{\mathsf{V}}
 \def\mVD{\mathsf{V}_{\mathsf{D}}}
 \def\mE{\mathsf{E}}

 \def\mv{\mathsf{v}}
 \def\me{\mathsf{e}}

\newcommand{\ee}{\mathrm{e}}

\newcommand{\fa}{{\mathfrak a}}

\newcommand{\fE}{{\mathcal E}}

\DeclareMathOperator{\Tr}{Tr} 
\DeclareMathOperator{\dd}{d\!} 
\newcommand{\PP}{\mathsf{P}} 
\newcommand{\ii}{\mathrm i}  
\DeclareMathOperator{\Dom}{{\mathcal D}} 
\DeclareMathOperator{\Ran}{\mathrm{Ran}} 

\DeclareMathOperator{\Span}{span}

\newcommand{\internal}{\mathrm{int}}
\newcommand{\external}{\mathrm{ext}}
\newcommand{\loc}{\mathrm{loc}}
\newcommand{\bdSp}{\cH}
\newcommand{\bdEv}{\Psi}
\newcommand{\preboundary}{{\mathcal{B}}}
\newcommand{\gauge}{U}
\newcommand{\gaugeBoundary}{V}
\newcommand{\boundarySign}{S}
\newcommand{\llower}{\ell^{\downarrow}}
\newcommand{\lupper}{\ell^{\uparrow}}
\newcommand{\eps}{\varepsilon}

\begin{document}

\title[Global bounds by small control sets]{Sturm-Liouville Problems And Global Bounds By Small Control Sets And applications to quantum graphs
}
\author[M.~Egidi]{Michela Egidi}
\address[M.~Egidi]{Universit\"at Rostock, Institut f\"ur Mathematik, D-18051 Rostock, Germany}
\email{michela.egidi@uni-rostock.de}

\author[D. Mugnolo]{Delio Mugnolo}
\address[D.~Mugnolo]{Lehrgebiet Analysis, Fakult\"at Mathematik und Informatik, Fern\-Universit\"at in Hagen, D-58084 Hagen,
Germany}
\email{delio.mugnolo@fernuni-hagen.de}

\author[A.~Seelmann]{Albrecht Seelmann}
\address[A.~Seelmann]{Technische Universit\"at Dortmund, Fakult\"at f\"ur Mathematik, D-44221 Dortmund, Germany}
\email{albrecht.seelmann@math.tu-dortmund.de}



\thanks{
The work of A.S.\ has been partially supported by the DFG grant VE~253/10-1 entitled \emph{Quantitative unique continuation
properties of elliptic PDEs with variable 2nd order coefficients and applications in control theory, Anderson
localization, and photonics}. The work of D.M.\ was partially supported by the Deutsche Forschungsgemeinschaft (Grant 397230547).
}

\begin{abstract}
	We develop a Logvinenko--Sereda theory for one-dimensional vector-valued self-adjoint operators. We thus deliver upper bounds on
	$L^2$-norms of eigenfunctions  -- and linear combinations thereof --  in terms of their $L^2$- and $W^{1,2}$-norms on small
	control sets that are merely measurable and suitably distributed along each interval. An essential step consists in proving a
	Bernstein-type estimate for Laplacians with rather general vertex conditions. Our results carry over to a large class of
	Schrödinger operators with magnetic potentials; corresponding results are unknown in higher dimension. We illustrate our
	findings by discussing the implications in the theory of quantum graphs.
\end{abstract}

\maketitle
\tableofcontents

\section{Introduction}

A typical assignment in control theory of quantum graphs is to bound the $L^2$-norm of a smooth function $f$ supported on a
metric graph $\Graph$ in terms of the norm of its restriction to a (typically, disconnected) \emph{control set} $\omega$. We here
are interested in finding sufficient conditions on $f$ and $\omega$ implying the inequality
\begin{equation}\label{eq:estim-base}
	\norm{f}{L^2(\Graph)}^2
	\leq
	C_\omega \norm{\chi_\omega f}{L^2(\Graph)}^2
	,
\end{equation}
where $\chi_{\omega}\in L^\infty(\Graph)$ is the characteristic function of $\omega$, and where the constant $C_\omega$ shall be
uniform in $f$ belonging to appropriate classes of functions. More specifically, we focus on the case where $f$ are
eigenfunctions -- or, more generally, linear combinations thereof -- of self-adjoint realisations of free -- or even magnetic --
Laplacians on metric graphs of semi-bounded geometry, i.e., such that the edges' length do not accumulate at $0$; we refer to the
monographs \cite{Pos12,BerKuc13,Mug14,Kur23}. Under this geometric assumption, the Laplacian is well known to be essentially
self-adjoint and the corresponding quantum graph can be equivalently regarded as a vector-valued Sturm--Liouville operator. This
motivates us to develop a control theory of one-dimensional second order operators on finite or countably infinite collections of
intervals that goes far beyond the metric graph setting.

Many local bounds on smooth functions on Euclidean domains are classical, including the  Harnack Inequality and  Hadamard's Three
Balls Theorem, see \cite{BerMal21} and references therein for a discussion of the interplay with spectral geometry. An
interesting way of proving \emph{pointwise} estimates for eigenfunctions of Schrödinger operators is based on the properties of
the \emph{torsion function} $u:=(-\Delta_D)^{-1}{\mathbf 1}$ of an open bounded domain $\Omega\subset \RR^d$. This was shown
in~\cite{FilMay12,Ber12} to be a convenient \textit{landscape function}, i.e., to allow for a pointwise bound
\begin{equation}\label{eq:filmay}
	\abs{\varphi(x)}
	\leq
	\abs{\lambda} \norm{\varphi}{\infty} u(x)
	,\quad
	x\in \Omega
	,
\end{equation}
for all eigenpairs $(\lambda,\varphi)$ of the Laplacian $\Delta_D$ with Dirichlet conditions on $\Omega$. Further landscape
functions that lead to sharper inequalities have been discovered ever since, cf.~\cite{Ste17,ArnDavFil19,Mug23}. These results
have inspired many investigations about localization properties of eigenfunctions of different classes of elliptic operators on
various geometric structures, including quantum graphs \cite{HarMal18,HarMal20,MugPlu23} and even general $M$-matrices
\cite{FilMayTao21}. Different but related (de)localization properties for eigenfunctions of combinatorial graphs have been
recently studied in~\cite{AnaSab19,LemSab20}.

In this work we focus on so-called \emph{Logvinenko--Sereda-type theorems}, which are currently available for Euclidean domains.
We extend them here to collections of one-dimensional intervals under general (possibly non-separated) self-adjoint boundary
conditions and, eventually, to possibly infinite metric graphs.
In its original form, the Logvinenko--Sereda Theorem goes back to \cite{Kacnelson-73, LogvinenkoS-74}, and gives a necessary and
sufficient geometric condition on the control set $\omega$ for an estimate of type \eqref{eq:estim-base} --with $\Graph$ replaced
by $\RR^{d}$ --  to be valid for functions with compactly supported Fourier transform. Roughly speaking, this condition requires
$\omega$ to be a measurable and well-distributed set in $\RR^d$. This result has been subsequently refined by Kovrijkine in
\cite{Kov01,Kovrijkine-thesis}, leading to an improved and qualitatively sharp constant in \eqref{eq:estim-base}.

Later on the very same technique of Kovrijkine has been adapted to prove an analogous estimate for functions with compactly
supported Fourier--Bessel transform, see~\cite{GhobberJ-13}. In more recent years his technique has found a broader application:
the functions he considers are elements in the range of the spectral projection of the Laplace operator on $\RR^d$ up to a
certain energy value, and it turns out that an estimate of type \eqref{eq:estim-base} is crucial in the theory of controllability
for the heat equation, see \cite{EgidiV-20} and the references therein. More generally, estimates of such a form for functions in
spectral subspaces of operators are of paramount importance in control theory and, consequently, a number of works have started
exploiting and adapting Kovrijkine's original ideas:
\cite{Egidi-21} deals with spectral subspaces of the Laplacian on infinite strips, 
\cite{BeauchardJPS-21} considers Hermite functions (i.e., functions in the spectral subspaces of the harmonic oscillator on
$\RR^d$),
\cite{MartinPS} obtains a spectral inequality for Hermite functions allowing $\omega$ to have holes of sublinearly growing
diameter, which is extended in \cite{DickeSVa,DickeSV-23} to also allow for a control set with decaying density,
\cite{Martin,AS} extend this further to general (anisotropic) Shubin operators, \cite{Martin22} treats functions in
Gelfand--Shilov spaces.
Finally, \cite{EgiSee21} proposes an abstract operator theoretical framework to derive such an inequality.
Extending the Logvinenko--Sereda Theorem to metric graphs paves the road to similar developments in spacial environments with
singularities, and may complement very recent results on controllability of parabolic equations on network-like structures:
\cite{Iwasaki-21} studying the observability of the heat equation with the standard Laplacian on equilateral graphs with finitely
many edges, all of finite length, and observability set being a large enough discrete set of points, \cite{MML21} treating the
observability of time-fractional diffusion equations on star graphs, and \cite{BarCavCoc21,AprBar22} considering parabolic
systems with elliptic second order operators on a metric tree and control sets being sets of leaves or open subgraphs, which may
or may not intersect all of the edges. 
The novelty of the techniques presented in this article lays in the fact that, while treating several realisation of the
(magnetic) Laplacian, we consider control sets that are merely measurable, although well-distributed in the considered space.

As we elaborate on the Logvinenko--Sereda approach, the skeleton of our work is based on \cite{EgiSee21}. However, the passage
from Euclidean domains to metric graphs that possibly feature edges of infinite length is not trivial: \cite{EgiSee21} relies
heavily on complex analytical tools and, in particular, on a unique continuation principle for complex-valued analytic functions
of several variables, which are not always available on metric graphs since the unique continuation principle generally fails,
see Remark~\ref{rmk:unique-continuation} below, as well as \cite[Section~3.4]{BerKuc13} and \cite{PluTau21,Kur21} for an overview
of topological and metric conditions implying that eigenfunctions are supported everywhere; indeed, circumventing the failure of
the unique continuation principle is a common issue in spectral geometric investigations of quantum graphs. Moreover, one needs
to find a good extension of the original geometric condition for the subset $\omega\subset \Graph$ to deal simultaneously with
edges of finite and infinite length and such that the set is large enough with respect to the whole graph. While in the Euclidean
setting the set $\omega$ is chosen such that the measure of each intersection with a ball of a certain radius contained in the
domain has a suitable lower bound, in our setting we require, inspired by \cite[Proposition~3.1]{EgiSee21}, the existence of a
covering for every edge in $\Graph$ by intervals of a variable but bounded length and overlapping only at the boundary such that
the measure of the intersection of $\omega$ with each of these intervals has a suitable lower bound; see
Definition~\ref{def:gammarhosampling} below. Here, the use of a specific covering with intervals of variable length allows to
obtain estimates with favourable parameters and to treat sets with larger gaps in the interior of each edge. For further details
we refer the reader to Remark~\ref{rem:sampling}\,(3) and Example~\ref{exa:basic}\,(3) below.

When applying our methods, it will be crucial to assume a function $f$ -- or a (spectral) class thereof -- to satisfy a Bernstein
inequality of the form
	\begin{equation}\label{eq:bernstein-intro}
		\norm{ f^{(m)} }{L^2(\Graph)}^2
		\leq
		C_B(m) \norm{f}{L^2(\Graph)}^2
		\quad \text{ for all }
		m\in\NN_0
	\end{equation}
with a sequence $(C_B(m))_{m\in \NN_0}$ fulfilling a suitable summability assumption. Condition~\eqref{eq:bernstein-intro}
is generally not expected to be satisfied for Schrödinger operators with nontrivial (non-analytic) potential $V$, but
could be treated on the whole Euclidean space $\RR^d$ for the harmonic oscillator (with quadratic potential) in
\cite{BeauchardJPS-21,EgiSee21} and for more general Shubin operators, including anharmonic oscillators with potential
$\abs{x}^{2k}$, in \cite{Martin,AS}.
Our results here are complementary to those in \cite[Section~4]{HarMal20}, which deliver pointwise estimates
\[
	\abs{\varphi(x)}
	\leq
	\kappa(x) \norm{\varphi\chi_\omega}{L^2(\Graph)}
\]
for suitable subgraphs  $\omega\subset \{x\in \Graph:V(x)\ge \lambda\}$ -- the complement of a ``potential well'' --, where
$\kappa$ is a function that depends on the internal structure of the metric graph.

Another prominent feature of our methods is that a handy factorization of a large class of self-adjoint realisations of the
magnetic Laplacian on metric graphs is available (see \cite[Section 2]{HundertSimon-03} for a similar factorization on domains,
and \cite[Section~1.4.1]{BerKuc13} or \cite[Section~2.2.2]{Pos12} for the case of the free Laplacian on metric graphs). We will
make good use of it to provide bounds on eigenfunctions for all such realisations. To the best of our knowledge, a direct
counterpart of our estimates for the Laplacian on domains is available in the literature only for Dirichlet, Neumann, and -- if
applicable -- periodic boundary conditions.

As already mentioned, we develop our theory for fairly general Sturm--Liouville-type operators with (separated or non-separated)
self-adjoint boundary conditions. Yet, it is especially charming to apply our theory to metric graphs, which we regard as
collections of intervals glued at their endpoints in an appropriate way that is encoded by suitable, non-separated boundary
conditions. For the \emph{standard Laplacian} (that is, the Laplacian realisation with continuity and Kirchhoff-type conditions
in the vertices; the most common since the pioneering investigations in~\cite{PavFad83,Nic87}), combining our general bounds with
known eigenvalue estimates of spectral geometric flavour we obtain results of the following kind:

\begin{quote}
\emph{Given a compact metric graph $\Graph$, then for finite linear combinations $f$ of eigenfunctions for the standard Laplacian
$-\Delta$ associated with the $m$ lowest eigenvalues we can prove that
\[
	\norm{\chi_\omega f}{L^2(\Graph)}^2
	\geq
	C\norm{f}{L^2(\Graph)}^2
\]
and
\[
	\norm{\chi_\omega f'}{L^2(\Graph)}^2
	\geq
	C\norm{f'}{L^2(\Graph)}^2
	,
\]
where $\omega$ can be a subgraph of $\Graph$, or even a general measurable subset of the metric measure space $\Graph$, as long
as it is sufficiently well-distributed among the edges. This yields, in particular, a corresponding estimate for the
$W^{1,2}$-norm,
\begin{equation}\label{eq:fh1fl2}
	\norm{ f }{W^{1,2}(\omega)}
	\geq
	C\norm{ f }{W^{1,2}(\Graph)},
\end{equation}
that seems to have no known counterpart in the case of domains.}
\end{quote}
We stress that the constant $C$ in~\eqref{eq:fh1fl2} is explicit and only depends on $m$, the subgraph $\omega$, and rough
information about the topological and metric structure of the metric graph -- more precisely, its  Betti number and diameter --,
see Corollary~\ref{cor:appl-metr-gr} below.

In particular, our bounds \eqref{eq:fh1fl2} are uniform with respect to the control set $\omega$ and therefore perform better in
the ``semi-supervised  case'', if an educated guess allows the user to place $\omega$ in a region where the function to be
estimated is presumably small. Note that in the last years, much heuristic insight above eigenfunction profiles on metric graphs
has been gained, see, e.g., \cite{BerKenKur19,BorCorJon21,KenRoh21,Mug23}.

The plan of this article is as follows: The general setting we are working with, including the description of suitable
realisations of the one-dimensional magnetic Laplacian and more general vector-valued Sturm--Liouville operators, is briefly
recalled in Section~\ref{ssec:framework}. In Section~\ref{ssec:mainresults} we introduce the main geometric notion of the
article -- that of sampling subsets of a metric graph, Definition~\ref{def:gammarhosampling} -- and then formulate
Theorem~\ref{thm:Laplacian} about global estimates for eigenfunctions of a general class of such realisations. In turn, this is a
more or less direct consequence of Theorem~\ref{thm:main}, which holds for all functions satisfying a Bernstein inequality (see
Definition~\ref{def:Bernstein}). Before turning to the proof of our main results, we substantiate in Section~\ref{sec:optimality}
the claim that our estimates are qualitatively optimal, with respect to the relevant parameters
(Example~\ref{exa:qualitatively-optimality}), and then discuss in Section~\ref{sec:appllications-semig} two applications to
parabolic problems associated with vector-valued Sturm--Liouville operators.

Section~\ref{sec:main_proof} is devoted to the proof of Theorem~\ref{thm:main}, which is subdivided in several lemmata. In
Section~\ref{sec:proof-of-main} we finally prove that the elements of spectral subspaces associated with a large class of
self-adjoint realisations of the magnetic Laplacian with general (possibly non-separated) self-adjoint boundary conditions forms
another, and more substantial, class of functions that satisfy Bernstein inequalities, thus completing the proof of
Theorem~\ref{thm:Laplacian}.

Finally, we specialize our findings and study the localization property of eigenfunctions of the standard Laplacian on metric
graphs. This is arguably the most interesting class of non-separated self-adjoint realisations of the free Laplacian. After
presenting a first (simple but non-trivial) example of a function satisfying the Bernstein inequality
(Example~\ref{exa:torsion}), namely the torsion function briefly discussed above, we present in Corollary~\ref{cor:appl-metr-gr}
an application of our theory to elements of spectral subspaces of the standard Laplacian.

Checking the Bernstein condition for self-adjoint realisations of the magnetic Laplacian -- the crucial step in the proof of
Theorem~\ref{thm:Laplacian} -- is based on a technically somewhat involved formalism  that allows us to describe when higher
order operators on metric graphs are powers of self-adjoint magnetic Laplacians. Such operators have been studied since
\cite{KosSch03}, see especially~\cite{Pan06c,Kur10,BerWey12}. For the sake of self-containedness, we recall some basic aspects of
their theory in the Appendix -- Section~\ref{sec:elliptic-metric}.

\section{Main results}\label{sec:notationAndMainResults}

In this section, we present the main results of the present work, along with the notational framework and the main geometric
concept for control subsets.

\subsection{Basic framework}\label{ssec:framework}
Let $\mE$ be a finite or countably infinite set. We consider a family $(\ell_\me)_{\me\in\mE} \subset (0,\infty]$ and set
$I_\me := \overline{[0,\ell_\me)}$. We refer to each $\me$ as an \emph{edge}, and to $\ell_\me$ as its \emph{length}. Note that
$I_\me = [0,\ell_\me]$ if $\ell_\me < \infty$ and $I_\me = [0,\infty)$ if $\ell_\me = \infty$. In particular, we allow for edges
of infinite length, i.e., $\ell_\me = \infty$.

If $\ell_\me\equiv \ell$, then in view of the isomorphism $L^2((0,\ell))\otimes \CC^\mE\simeq \bigoplus_{\me\in\mE} L^2((0,\ell))$
we can regard vector-valued Sturm--Liouville problems as a family (indexed in $\mE$) of scalar-valued Sturm--Liouville problems.
We are, however, mainly interested in the general case of edges of possibly different lengths. We therefore focus right away on
families of functions $f_\me \colon I_\me\to \CC$, $\me \in \mE$, which can, by construction, be
identified with functions $f \colon \cE \to \CC$, where
\begin{equation}\label{eq:defEdgeSet}
	\cE
	:=
	\bigsqcup\limits_{\me\in\mE} I_\me
	.
\end{equation}
We may then consider the (canonically defined) function space $L^2(\cE) = \bigoplus_{\me\in\mE} L^2((0,\ell_\me))$ endowed with
the norm
\[
	\norm{f}{L^2(\cE)}
	=
	\biggl( \sum_{\me\in\mE} \norm{f_{\me}}{L^2((0,\ell_\me))}^2 \biggr)^{1/2}
	.
\]
We also introduce
\[
	L_\loc^1(\cE)
	:=
	\{ f \colon \cE \to \CC \mid f_\me \in L_\loc^1(I_\me) \quad\forall \me \in \mE \}
\]
and
\[
	W_\loc^{1,1}(\cE)
	:=
	\{ f \colon \cE \to \CC \mid f_\me \in W_\loc^{1,1}(I_\me) \quad\forall \me \in \mE \}
	.
\]

We impose the following hypothesis.
\begin{assum}\label{assum-main}
	The edge set $\cE$ is of \emph{semi-bounded geometry}, that is,
	\[
	\llower
	:=
	\inf_{\me \in \mE} \ell_\me
	>
	0
	.
	\]
\end{assum}

With the aim of parametrising different self-adjoint realisations of second order differential operators, it is useful to
distinguish between edges of finite and infinite length, which we call \emph{internal} and \emph{external edges},
respectively. We therefore set
\[
	\mE_\internal
	:=
	\{ \me \in \mE \mid \ell_\me < \infty \}
	\quad \text{ and }\quad  
	\mE_\external
	:=
	\mE \setminus \mE_\internal
	=
	\{ \me \in \mE \mid \ell_\me = \infty \}
	.
\]
Under the Assumption~\ref{assum-main}, for every real-valued $A \in L_\loc^1(\cE)$ we may consider the self-adjoint realisation
$\Delta_{A,Y}$ of the \textit{magnetic Laplacian} in $L^2(\cE)$ \textit{associated with a closed subspace} $Y$ of
$\ell^2(\mE) \oplus \ell^2(\mE_\internal)$. More precisely, setting
\[
	W_A(\cE)
	:=
	\{ f \in L^2(\cE) \cap W_\loc^{1,1}(\cE) \mid \ii f' + Af \in L^2(\cE) \}
\]
with $f' = (f'_\me)_{\me\in\mE}$, the operator $\Delta_{A,Y}$ in $L^2(\cE)$ defined by
\begin{equation}\label{eq:defMagLap}
	\begin{aligned}
		\Dom(\Delta_{A,Y})
		&=
		\bigl\{ f \in W_A(\cE) \mid \ii f' + Af \in W_A(\cE)
		,\
		\bdEv_+(f) \in Y,\ \bdEv_-(\ii f'+Af) \in Y^\perp \bigr\},\\
		\Delta_{A,Y} f
		&=
		-\ii(\ii f' + Af)' - A(\ii f' + Af)
		,\quad
		f \in \Dom(\Delta_{A,Y})
		,
	\end{aligned}
\end{equation}
is self-adjoint and non-positive, where
\[
	\bdEv_\pm(g)
	=
	(\pm g_\me(0))_{\me\in\mE} \oplus (g_\me(\ell_\me))_{\me\in\mE_\internal}
	,\quad
	g \in W_A(\cE)
	,
\]
and where $Y^\perp$ denotes the orthogonal complement of $Y$ in $\ell^2(\mE) \oplus \ell^2(\mE_\internal)$; note that by Sobolev
embedding, for $g = (g_\me)_{\me\in\mE} \in W_A(\cE)$ each $g_\me$ can be identified with a continuous function on $I_\me$. The
particular case of $A = 0$ gives corresponding realisations of the free Laplacian and covers (upon a suitable choice of $Y$) those
with separated boundary conditions, including Dirichlet, Neumann, or mixed conditions on each edge.
Moreover, in the context of metric graphs as mentioned in the introduction, this also covers (for $A = 0$) the
\emph{standard Laplacian} with continuity and Kirchhoff-type vertex conditions, as well as realisations with anti-Kirchhoff vertex
conditions, see Example~\ref{exa:vert-cond} below.

A more detailed review of the magnetic Laplacians $\Delta_{A,Y}$ is presented in Section~\ref{sec:proof-of-main} and
Appendix~\ref{sec:elliptic-metric} below. Also, a related technical result to study the powers {$\Delta_{A,Y}^m$} needed for our
main results is discussed in Proposition~\ref{prop:powerLaplace} and may be of independent interest.

We aim for estimates of the form
\begin{equation}\label{eq:aim}
	\norm{ \chi_\omega f }{L^2(\cE)}^2
	\geq
	C(\omega) \norm{f}{L^2(\cE)}^2
\end{equation}
for certain classes of functions in $L^2(\cE)$ and certain subsets $\omega \subset \cE$, where $\chi_\omega$ denotes the
characteristic function of $\omega$. Here, one usually expects that the constant $C(\omega)$ can be chosen the larger the more
well-spread the set $\omega$ is along the edges. Moreover, estimates of the above form imply that $f$ already has to vanish on the
whole $\cE$ if it vanishes on $\omega$. Therefore, the portion of $\omega$ on each edge has to be large enough in comparison with
the edge itself, equivalently, $\omega$ has to be ``well-distributed'' in $\cE$. The following geometric definition formalises
this idea and provides an adaptation of the corresponding geometric condition for $\RR^d$, first appearing in
\cite{LogvinenkoS-74,Kacnelson-73}. Here and in the following, we denote by $\abs{\cdot}$ the usual $1$-dimensional Lebesgue
measure.

\begin{defin}\label{def:gammarhosampling}
	A measurable set $\omega_\me \subset I_\me$ is called \emph{$(\gamma,\rho)$-sampling in $I_\me$} for some $\gamma \in (0,1]$ and
	some $\rho>0$ if there is a finite or countably infinite family $(J_{\me,k})_k$ of closed intervals $J_{\me,k} \subset I_\me$
	such that
	\begin{itemize}
		
		\item
		$\bigcup_k J_{\me,k}=I_\me$,
		
		\item
		the intervals $J_{\me,k}$ have mutually disjoint interior,
		
		\item
		the length of each $J_{\me,k}$ is at most $\rho$,
		
		\item
		$\abs{ \omega_\me \cap J_{\me,k} } \geq \gamma \abs{J_{\me,k}}$ for all $k$.
		
	\end{itemize}
	We say that $\omega = \bigsqcup_{\me\in\mE}\omega_\me \subset \cE$ is \emph{$(\gamma,\rho)$-sampling in $\cE$} if each
	$\omega_{\me} \subset I_\me$ is $(\gamma,\rho)$-sampling in $I_\me$.
\end{defin}

Examples and properties of sampling sets are discussed after the main results below.

\subsection{Main results}\label{ssec:mainresults}
Our first main result now establishes estimates of \sout{the} type \eqref{eq:aim} for sampling sets $\omega$ and for functions and
their magnetic derivatives in spectral subspaces $\Ran\PP_{-\Delta_{A,Y}}(\lambda)$ for magnetic Laplacians $\Delta_{A,Y}$, where
$\PP_{-\Delta_{A,Y}}(\lambda)$ denotes the spectral projection for $-\Delta_{A,Y}$ associated to the energy interval
$(-\infty,\lambda]$. The corresponding constant $C(\omega)$ then depends on the sampling parameters of the set $\omega$ and the
energy level $\lambda$.

\begin{thm}\label{thm:Laplacian}
	Let $\cE$ satisfy Assumption~\ref{assum-main}, and let $\Delta_{A,Y}$ be defined as in \eqref{eq:defMagLap} with some
	real-valued $A \in L_\loc^1(\cE)$ and some closed subspace $Y \subset \ell^2(\mE) \oplus \ell^2(\mE_\internal)$. Then, for every
	function $f\in \Ran\PP_{-\Delta_{A,Y}}(\lambda)\setminus \{0\}$, $\lambda \geq 0$, and every $(\gamma,\rho)$-sampling set
	$\omega \subset \fE$ we have
	\begin{equation}\label{eq:eigenfunctions}
		\norm{\chi_\omega f}{L^2(\cE)}^2
		>
		12\Big(\frac{\gamma}{48}\Big)^{\frac{40\rho\sqrt{\lambda}}{\log 2}+5} \norm{f}{L^2(\cE)}^2
		,
	\end{equation}
	as well as
	\begin{equation}\label{eq:eigenfunctionsDer}
		\norm{\chi_\omega (\ii f' + Af)}{L^2(\cE)}^2
		>
		12\Big(\frac{\gamma}{48}\Big)^{\frac{40\rho\sqrt{\lambda}}{\log 2}+5} \norm{(\ii f' + Af)}{L^2(\cE)}^2
		.
	\end{equation}
\end{thm}

It is worth to note that if $\Delta_{A,Y}$ has purely discrete spectrum, then the functions in the spectral subspace
$\Ran\PP_{-\Delta_{A,Y}}(\lambda)$ under consideration in Theorem~\ref{thm:Laplacian} are just finite linear combinations of
eigenfunctions associated to eigenvalues not exceeding $\lambda$. This is the case precisely if $\Dom(\Delta_{A,Y})$
(equivalently, $\Dom((-\Delta_{A,Y})^{1/2})$) is compactly embedded in $L^2(\cE)$ with respect to the corresponding graph norm.

Theorem~\ref{thm:Laplacian} is a consequence of a more general result for those functions in
\[
W^{\infty,2}(\cE)
:=
\bigcap_{m\in\NN} W^{m,2}(\cE)
:=
\bigcap_{m\in\NN} \bigoplus_{\me \in \mE} W^{m,2}((0,\ell_\me))
,
\]
where $W^{m,2}((0,\ell_\me))$ is the usual $L^2$-Sobolev space of order $m$ on the interval $(0,\ell_\me)$, satisfying a
Bernstein-type inequality in the following sense.

\begin{defin}\label{def:Bernstein}
	We say that $f \in W^{\infty,2}(\cE)$ satisfies a \emph{Bernstein-type inequality with respect to
		$C_B \colon \NN_0 \to [0,\infty)$} if
	\begin{equation}\label{eq:bernstein}
		\norm{ f^{(m)} }{L^2(\cE)}^2
		\leq
		C_B(m) \norm{f}{L^2(\cE)}^2
		\quad \text{ for all }\
		m\in\NN_0
	\end{equation}
	with $f^{(m)} = (f_\me^{(m)})_{\me\in\mE}$.
\end{defin}

Provided that $f \neq 0$, the above definition clearly requires that for each $m \in \NN_0$ the constant $C_B(m)$ is larger or
equal to $\norm{f^{(m)}}{L^2(\cE)}^2 / \norm{f}{L^2(\cE)}^2$, and in some sense the optimal choice would be
with equality. However, allowing here for an inequality opens the way to an estimate of the form \eqref{eq:aim} that is uniform
over a class of functions. For instance, it turns out that suitable transformations of the functions in
$\Ran\PP_{-\Delta_{A,Y}}(\lambda)$ considered in Theorem~\ref{thm:Laplacian} satisfy such a Bernstein-type inequality with respect
to $C_B(m) = \lambda^m$, $m \in \NN_0$, see Corollary~\ref{cor:powerLaplace} below. 

\begin{thm}\label{thm:main}
	Let $\cE$ be as in \eqref{eq:defEdgeSet}, and suppose that $f \in W^{\infty,2}(\cE) \setminus \{0\}$ satisfies a
	Bernstein-type inequality with respect to $C_B \colon \NN_0 \to [0,\infty)$.
	
	If
	\begin{equation}\label{eq:BernsteinSum}
		h
		:=
		\sum_{m\in\NN_0} (C_B(m))^{1/2} \frac{(10\rho)^m}{m!}
		<
		\infty
	\end{equation}
	for some $\rho > 0$, then, for every $(\gamma,\rho)$-sampling set $\omega \subset \cE$ we have
	\begin{equation}\label{eq:BernsteinSum-2}
		\norm{\chi_\omega f}{L^2(\cE)}^2
		>
		12\Big(\frac{\gamma}{48}\Big)^{\frac{4\log h}{\log 2}+5} \norm{f}{L^2(\cE)}^2
		.
	\end{equation}
\end{thm}

It is worth to emphasise that Theorem~\ref{thm:main} holds under no additional assumptions on the geometry of $\cE$. In
particular, for the scope of Theorem~\ref{thm:main} we do \emph{not} need to impose Assumption~\ref{assum-main}.

Moreover, we point out that, although Theorem~\ref{thm:Laplacian} and our main applications in
Section~\ref{sec:appllications-semig} deal with classes of functions, Theorem~\ref{thm:main} is tailored towards individual
functions and applies, for instance, to polynomials, which trivially satisfy a Bernstein-type inequality with respect to an
eventually vanishing $C_B$. A less usual function satisfying a Bernstein-type inequality is the torsion function from the
introduction, which is edgewise a quadratic polynomial, see Example~\ref{exa:torsion} below.

\subsection{Discussion on sampling sets}\label{ssec:discussions}

We now discuss in more detail the notion of sampling sets from Definition~\ref{def:gammarhosampling}, as well as its relation to
our main results and further consequences.

We first collect some elementary observations regarding this notion.

\begin{rem}\label{rem:sampling}
	\begin{enumerate}[(1)]	
		
		\item
		The sets $\omega_\me$ in Definition~\ref{def:gammarhosampling} are only assumed to be measurable, but not necessarily to be
		open. For instance, each $\omega_\me$ may also be a fractal set of positive measure, cf.\ part \eqref{exa:volterra} of
		Example~\ref{exa:basic} below.
		
		\item
		The parameter $\gamma$ in Definition~\ref{def:gammarhosampling} measures the proportion of the measure of $\omega$ per portion
		of each edge. These portions are given in terms of the adjacent intervals of length less than or equal to $\rho$ that cover
		each edge.
		
		\item
		The smaller $\rho$ in Definition~\ref{def:gammarhosampling} (with fixed $\gamma$) the more well-distributed $\omega$ has to be
		along each edge. Indeed, $\rho$ determines an upper bound for the size of gaps $\omega$ is allowed to have on each edge. More
		precisely, it is easy to see that the length of gaps each $\omega_\me$ can have is at most $(1-\gamma)\rho$ at the endpoints
		and $2(1-\gamma)\rho$ in the interior.
		
		\item
		If $\omega \subset \cE$ is $(\gamma,\rho)$-sampling in $\cE$ for some $\rho > 0$ and $\gamma \in (0,1]$, then $\omega$ is also
		$(\gamma',\rho')$-sampling in $\cE$ for all $\rho' \geq \rho$ and $\gamma' \leq \gamma$. Hence, the prefactor of the norm in
		the right-hand sides of \eqref{eq:eigenfunctions}, \eqref{eq:eigenfunctionsDer}, and \eqref{eq:BernsteinSum-2}
		obviously gets larger the smaller $\rho$ and the larger $\gamma$ are. This is consistent with the discussion
		preceding Definition~\ref{def:gammarhosampling}.
		
		\item
		Suppose that for each $\me \in \mE$ the set $\omega_\me \subset I_\me$ is $(\gamma_\me,\rho_\me)$-sampling in $I_\me$ such that
		$\gamma := \inf_{\me\in\mE} \gamma_\me > 0$ and $\rho := \sup_{\me \in \mE} \rho_\me < \infty$. Then the set
		$\omega = \bigsqcup_{\me \in \mE} \omega_\me \subset \cE$ is $(\gamma,\rho)$-sampling in $\cE$.
		
		\item\label{it:rhoOptimal}
		There is always an optimal (i.e., minimal) choice of $\rho$ with respect to a fixed density $\gamma$. More precisely, if
		$\omega$ is $(\gamma,\rho)$-sampling in $\cE$, then $\omega$ is also $(\gamma,\tilde{\rho})$-sampling in $\cE$ with
		$\tilde{\rho} := \sup_{\me\in\mE} \sup_k \abs{J_{\me,k}} \leq \rho$, where $(J_{\me,k})_k$ denotes any family of adjacent
		intervals covering $I_\me$ consistent with the definition of $\omega$ being $(\gamma,\rho)$-sampling.
		
		\item
		Similarly as in \eqref{it:rhoOptimal}, there is always an optimal (i.e., maximal) choice of $\gamma$ with respect to a fixed
		scale $\rho$. More precisely, if $\omega$ is $(\gamma,\rho)$-sampling in $\cE$, then $\omega$ is also
		$(\tilde{\gamma},\rho)$-sampling in $\cE$ with $\tilde{\gamma} := \inf_{\me\in\mE} \tilde{\gamma}_\me \geq \gamma$, where
		\[
		\tilde{\gamma}_\me
		=
		\inf_k \frac{\abs{\omega_\me \cap J_{\me,k}}}{\abs{J_{\me,k}}}
		\geq
		\gamma
		.
		\]
		Here, $(J_{\me,k})_k$ again denotes any family of adjacent intervals covering $I_\me$ consistent with the
		definition of $\omega$ being $(\gamma,\rho)$-sampling.
		
	\end{enumerate}
\end{rem}

We now discuss several choices of sets $\omega$.

\begin{exa}\label{exa:basic}
	\begin{enumerate}[(1)]
		
		\item
		$\omega = \cE$ and $\omega = \bigsqcup_{\me\in\mE}(0,\ell_\me)$ are both $(1,\rho)$-sampling in $\cE$ for all $\rho > 0$.
		
		\item
		For $\mE = \NN$ and $\cE = \bigsqcup_{k\in\NN}[0,1]$, the choice $\omega_k = [0,1/k]$ leads to an $(1/k,1)$-sampling set in the
		corresponding edge. But $\inf_{k\in\NN} \abs{\omega_k} = 0$, so the set $\omega = \bigsqcup_{k\in\NN} \omega_k$ can \emph{not}
		be $(\gamma,\rho)$-sampling in $\cE$ for any choice of $\gamma$ and $\rho$.
		
		\item
		Suppose that $\ell_\me < \infty$ (that is, $\me \in \mE_\internal$), and consider the three sets
		$\omega_\me^{(1)} = (0,\ell_\me/2)$, $\omega_\me^{(2)} = (0,\ell_\me/4) \cup (3\ell_\me/4,\ell_\me)$, and
		$\omega_\me^{(3)} = (\ell_\me/4,3\ell_\me/4)$. All three are obviously $(1/2,\rho)$-sampling in $\me$ for $\rho \geq \ell_\me$
		with the trivial one interval covering. On the other hand, $\omega_\me^{(1)}$ and $\omega_\me^{(2)}$ both have a gap of length
		$\ell_\me/2$, whereas $\omega_\me^{(3)}$ has two gaps of length only $\ell_\me/4$, and they show a different behaviour with
		respect to sampling properties for $\rho < \ell_\me$: $\omega_\me^{(1)}$ is $(1-\ell_\me/(2\rho),\rho)$-sampling in $I_\me$ for
		$\ell_\me/2 < \rho < \ell_\me$ via $[0,\ell_\me] = [0,\ell_\me-\rho] \cup [\ell_\me-\rho,\ell_\me]$ and is \emph{not}
		$(\gamma,\rho)$-sampling in $I_\me$ for any choice of $\gamma$ for $\rho \leq \ell_\me/2$ due to the gap
		$[\ell_\me/2,\ell_\me]$.
		By contrast, both sets $\omega_\me^{(2)}$ and $\omega_\me^{(3)}$ are \emph{not} $(\gamma,\rho)$-sampling in $I_\me$ for any
		choice of $\gamma$ for $\rho \leq \ell_\me/4$ due to the gap $[\ell_\me/4,3\ell_\me/4]$ of length $\ell_\me/2$ in the interior
		of $\omega_\me^{(2)}$ and the gaps $[0,\ell_\me/4]$ and $[3\ell_\me/4,\ell_\me]$ of length $\ell_\me/4$ at the endpoints of
		$\omega_\me^{(3)}$, respectively. However, they are both $(1/2,\rho)$-sampling in $I_\me$ for $\ell_\me/2 \leq \rho < \ell_\me$
		via $[0,\ell_\me] = [0,\ell_\me/2] \cup [\ell_\me/2,\ell_\me]$, and $(1-\ell_\me/(4\rho),\rho)$-sampling in $I_\me$ for
		$\ell_\me/4 < \rho < \ell_\me/2$. The latter can be seen for $\omega_\me^{(2)}$ via the covering
		\[
		[0,\ell_\me]
		=
		[0,\ell_\me/2-\rho] \cup [\ell_\me/2-\rho,\ell_\me/2] \cup [\ell_\me/2,\ell_\me/2+\rho] \cup [\ell_\me/2+\rho,\ell_\me]
		,
		\]
		and for $\omega_\me^{(3)}$ via
		\[
		[0,\ell_\me]
		=
		[0,\rho] \cup [\rho,\ell_\me/2] \cup [\ell_\me/2, \ell_\me/2+\rho] \cup [\ell_\me-\rho,\ell_\me]
		.
		\]
		
		\item
		Suppose that $\ell_\me = \infty$ (that is, $\me \in \mE_\external$), and let $\omega_\me$ be a measurable $1$-periodic set with
		$\abs{ \omega_\me \cap [0,1] } = \gamma$ for some $\gamma \in (0,1)$. Then, $\omega_\me$ has gaps of length at most $1-\gamma$
		and is clearly $(\gamma,1)$-sampling in $I_\me$ via the covering $[0,\infty) = \bigcup_{k\in\NN_0} [k,k+1]$. In fact, for every
		measurable subset $A \subset [0,1]$ with measure $\abs{A} > 1-\gamma$ we have
		$\abs{\omega_\me \cap A} \geq \abs{A} - (1-\gamma)$. Hence, taking into account periodicity of $\omega_\me$ and using the
		covering $[0,\infty) = \bigcup_{k\in\NN_0} [k\rho,(k+1)\rho]$ with $\rho > 1-\gamma$, it is easy to see that $\omega_\me$ is
		$(1-(1-\gamma)/\rho,\rho)$-sampling in $I_\me$ for $1-\gamma < \rho \leq 1$, $(\gamma/\rho,\rho)$-sampling for
		$1 < \rho \leq 2-\gamma$, $(1-2(1-\gamma)/\rho,\rho)$-sampling for $2-\gamma < \rho \leq 2$, and so forth. In particular,
		$\omega_\me$ is $(\gamma/(2-\gamma),\rho)$-sampling in $I_\me$ for all $\rho \geq 1$.
		
		\item\label{exa:volterra}
		Suppose that $\ell_\me=1$ and let $\omega_\me$ be the Smith--Volterra--Cantor set, i.e., the set obtained by successively
		removing intervals of length $1/4^n$ from the middle of each of the previously obtained $2^{n-1}$ intervals. This set is closed
		with empty interior and has measure $1/2$. Since the largest gap of $\omega_\me$ has length $1/4$, it fails to be
		$(\gamma,\rho)$-sampling for any choice of $\gamma$ if $\rho\in(0,1/8]$. However, it is $(1/2,1/2)$-sampling via
		$[0,1]=[0,1/2]\cup[1/2,1]$, as each intervals contains half of $\omega_\me$, as well as $(4/9,9/32)$-sampling via
		$[0,1]=[0,7/32] \cup [7/32,1/2] \cup [1/2,25/32] \cup [25/32,1]$, as each interval contains a portion of $\omega_\me$ of
		measure $1/8$.
		
	\end{enumerate}
\end{exa}

We conclude the section by commenting on a possible extension of Definition~\ref{def:gammarhosampling} and its
influence on the results.

\begin{rem}\label{rem:extension}
	It is possible to extend Theorem~\ref{thm:main} by relaxing the notion of sampling sets to allow the intervals $J_{\me,k}$ to
	mutually overlap. For some sets, this could result in slightly preferable parameters $\gamma$ and $\rho$ (by factors $2$ and
	$1/2$, respectively), but the proof would become more technical and, if there was no overlap of, say, three or more intervals,
	\eqref{eq:BernsteinSum-2} would have to be replaced by
	\[
		\norm{\chi_\omega f}{L^2(\cE)}^2
		>
		6\Big(\frac{\gamma}{48}\Big)^{\frac{4\log h}{\log 2}+7} \norm{f}{L^2(\cE)}^2
		,
	\]
	see Remark~\ref{rem:overlap} below. However, since we are interested mainly in the general behaviour of the estimate in terms of
	$C_B(m)$ (resp.\ $\lambda$) and the geometric parameters of $\omega$, we opted for the current simpler version of
	Definition~\ref{def:gammarhosampling}, and consequently of estimate \eqref{eq:BernsteinSum-2}.
\end{rem}

\section{Optimality of the main results}\label{sec:optimality}

We now analyse more closely Theorems~\ref{thm:Laplacian} and~\ref{thm:main}. We start with a discussion about the optimality of
the estimate in Theorem~\ref{thm:Laplacian} with respect to the behaviour in terms of $\gamma^{\rho\sqrt{\lambda}}$, which is
essentially already known from the Euclidean setting, see, e.g., \cite[Example~3.1]{EgidiV-20}.

\begin{exa}\label{exa:qualitatively-optimality}
	Let us consider the second derivative on an interval $(0,\ell)$, $\ell < \infty$, with Neumann boundary conditions. These fit
	into our framework of Theorem~\ref{thm:Laplacian} upon taking $\#\mE = 1$, $\cE = [0,\ell]$, and
	$Y = \ell^2(\mE)\times \ell^2(\mE_\internal)\equiv \CC^2$.
	
	Consider for some $\gamma\in (0,4/\pi^2]$ the subset
	$\omega = [\frac{\ell}{4}(1-\gamma), \frac{\ell}{4}(1+\gamma)] \cup [\frac{\ell}{4}(3-\gamma), \frac{\ell}{4}(3+\gamma)] \subset
	\cE$, which, in light of the covering $[0,\ell] = [0,\ell/2] \cup [\ell/2,\ell]$, is $(\gamma,\ell/2)$-sampling in $\cE$.
	We fix some large enough $\lambda>0$ such that $\alpha := \lfloor \frac{\ell\sqrt{\lambda}}{2\pi} \rfloor \geq 2$, and take
	\[
		f(x)
		:=
		\cos^\alpha \Bigl( \frac{2\pi x}{\ell} \Bigr)
		,\quad
		x \in \cE
		.
	\]
	It is easy to see, by Fourier analysis, that $f$ is a linear combination of eigenfunctions for the second derivative associated
	to eigenvalues less then or equal to $\frac{4\pi^2\alpha^2}{\ell^2} \leq \lambda$ and, thus, belongs to
	$\Ran\PP_{-\Delta_{0,Y}}(\lambda) \setminus \{0\}$ with $Y=\CC^2$.
	
	By Jensen's inequality we have
	\[
		\norm{f}{L^2(\cE)}^2
		=
		\int_0^{\ell} \cos^{2\alpha}\Bigl( \frac{2\pi x}{\ell} \Bigr) \,\dd x 
		\geq
		\ell \biggl( \int_0^{\ell} \Bigl\lvert \cos \Bigl( \frac{2\pi x}{\ell} \Bigr) \Bigr\rvert \frac{1}{\ell} \,\dd x \biggr)^{2\alpha} 
		=
		\ell \Bigl( \frac{2}{\pi} \Bigr)^{2\alpha}
		.
	\]
	Using the symmetry of the cosine function, the fact that $\cos(x)\leq -x+\pi/2$ for $x \in [0,\pi/2]$, and a change of variable,
	we also estimate	
	\begin{align*} 
		\norm{\chi_\omega f}{L^2(\cE)}^2
		&=
		\int_{\omega} \cos^{2\alpha} \Bigl( \frac{2\pi x}{\ell} \Bigr) \,\dd x
		=
		2\int_{\frac{\ell}{4}(1-\gamma)}^{\frac{\ell}{4}(1+\gamma)} \cos^{2\alpha}\Bigl( \frac{2\pi x}{\ell} \Bigr) \,\dd x\\
		&=
		\frac{\ell}{\pi}\int_{\frac{\pi}{2}(1-\gamma)}^{\frac{\pi}{2}(1+\gamma)} \cos^{2\alpha}(y) \,\dd y
		=
		\frac{2\ell}{\pi}\int_{\frac{\pi}{2}(1-\gamma)}^{\frac{\pi}{2}} \cos^{2\alpha}(y) \,\dd y\\
		&\leq
		\frac{2\ell}{\pi}\int_{\frac{\pi}{2}(1-\gamma)}^{\frac{\pi}{2}} \Bigl( -y +\frac{\pi}{2} \Bigr)^{2\alpha} \,\dd y
		=
		\frac{2\ell}{\pi(2\alpha+1)}\Bigl( \frac{\pi \gamma}{2} \Bigr)^{2\alpha +1}
		. 
	\end{align*} 
	In light of $\alpha\geq 2$, $2\alpha +1 = 2 \lfloor \ell\sqrt{\lambda}/(2\pi) \rfloor +1 \geq \ell\sqrt{\lambda}/\pi -1$, and
	$\gamma \leq 4/\pi^2$, we conclude that
	\[
		\frac{\norm{\chi_\omega f}{L^2(\cE)}^2}{\norm{f}{L^2(\cE)}^2}
		\leq
		\frac{1}{\ell}\Bigl(\frac{\pi}{2}\Bigr)^{2\alpha} \frac{2\ell}{\pi(2\alpha+1)} \Bigl(\frac{\pi \gamma}{2}\Bigr)^{2\alpha +1}
		\leq
		\frac{1}{5} \Bigl( \frac{\pi^2\gamma}{4} \Bigr)^{2\alpha+1}
		\leq
		\frac{1}{5} \Bigl( \frac{\pi^2}{4}\gamma \Bigr)^{\frac{\ell\sqrt{\lambda}}{\pi}-1}
		,
	\]
	which is consistent with the lower bound obtained in Theorem~\ref{thm:Laplacian} with respect to the behaviour of the form
	$\gamma^{\rho\sqrt{\lambda}}$. Thus, asymptotically speaking, one cannot expect in Theorem~\ref{thm:Laplacian} an estimate with
	qualitatively better behaviour.
\end{exa}

Although Example~\ref{exa:qualitatively-optimality} shows that the estimate from Theorem~\ref{thm:Laplacian} is (qualitatively)
optimal, the estimate does not perform well in certain situations, as the examples below show.

\begin{exa}
	\begin{enumerate}[(1)]
		
		\item
		Consider again the set $\omega = \cE$, which is $(1,\rho)$-sampling in $\cE$ for all $\rho > 0$, see item (1) of
		Example~\ref{exa:basic}. Then, for every $A\in L^1_\loc(\cE)$ and every closed subspace $Y$ of
		$\ell^2(\mE)\times \ell^2(\mE_\internal)$, taking the limit as $\rho \to 0$ in Theorem~\ref{thm:Laplacian} gives
		\begin{equation}\label{eq:eigenfunctions-exa}
			\frac{\norm{\chi_\omega f}{L^2(\cE)}^2}{\norm{f}{L^2(\cE)}^2}
			\geq
			12\Big(\frac{1}{48}\Big)^5 
			\approx
			4.7\cdot 10^{-8}
			,\quad
			f \in \Ran\PP_{-\Delta_{A,Y}}(\lambda) \setminus \{0\},\ \lambda \geq 0
			.
		\end{equation}	
		Since, on the other hand, $\norm{\chi_\omega f}{L^2(\cE)}^2 / \norm{f}{L^2(\cE)}^2 = 1$, this indicates that the estimate does
		not perform well numerically, at least if $\omega$ is ``large'' in $\cE$.
		
		\item
		Consider on $(0,\pi)$ the ($L^2$-normalised) eigenfunction $f(x)=\sqrt{\frac{2}{\pi}}\cos(kx)$ for the second derivative with
		Neumann boundary conditions associated to the eigenvalue $\lambda = k^2$ with $k \in \NN$. Moreover, with $\#\mE = 1$ and
		$\cE = [0,\pi]$, let $\omega = ( \pi(1-\gamma)/2, \pi(1+\gamma)/2 ) \subset \cE$ for some $\gamma \in (0,1)$, which in light of
		$[0,\pi] = [0,\pi/2] \cup [\pi/2,\pi]$ is $(\gamma, \pi/2)$-sampling in $\cE$. An elementary calculation then shows that
		\begin{align*}
			\frac{\norm{\chi_\omega f}{L^2(\cE)}^2}{\norm{f}{L^2(\cE)}^2}
			&=
			\norm{\chi_\omega f}{L^2(\cE)}^2
			=
			\frac{2}{\pi}\int_\omega \cos^2(kx) \,\dd x
			=
			\frac{4}{\pi}\int_{\frac{\pi}{2}(1-\gamma)}^{\frac{\pi}{2}} \cos^2(kx) \,\dd x\\
			&=
			\frac{\pi\gamma}{4} + \frac{2\cos(\frac{k\pi( \gamma-1)}{2})\sin(\frac{k\pi( \gamma-1)}{2}) +
				2\cos(\frac{k\pi}{2})\sin(\frac{k\pi}{2})}{4k}
			.
		\end{align*}
		The latter converges to $\pi\gamma/4 > 0$ as $k \to \infty$, whereas the lower bound provided by Theorem~\ref{thm:Laplacian}
		converges to zero exponentially as $k \to \infty$.
		
	\end{enumerate}
\end{exa}

\section{Applications}\label{sec:appllications-semig}

In this section we discuss two applications of our main results: a bound for the trace of the semigroup
$(\ee^{t\Delta_{A,Y}})_{t\geq 0}$ and observability/null-controllability of
the heat equation on $L^2(\cE)$.

\subsection{A bound for the trace}

Recall that the heat semigroup generated by all realisations of the magnetic Laplacian whose form domain is a closed subspace of
$W_A(\cE)$ is of trace class whenever the set $\mE$ is finite with $\mE_\external=\emptyset$. Indeed, due to analyticity of the
semigroup, the exponential $\ee^{\frac{t}{2}\Delta_{A,Y}}$ is for each $t > 0$ bounded as an operator from $L^2(\cE)$ to
$W_A(\cE)$, see, e.g., \cite[Proposition 7.3.4]{Arendt-ISEM}. In turn, due to the ideal property of Hilbert--Schmidt operators,
and because the embedding of $W_A(\cE)$ into $L^2(\cE)$ is Hilbert--Schmidt (see Lemma~\ref{lem:magneticSobolev} below), the
exponential $\ee^{\frac{t}{2}\Delta_{A,Y}}$ is for each $t>0$ likewise Hilbert--Schmidt. We therefore conclude that
$\ee^{t\Delta_{A,Y}}=\ee^{\frac{t}{2}\Delta_{A,Y}}\ee^{\frac{t}{2}\Delta_{A,Y}}$ is of trace class for all $t>0$.

The following can now be checked directly considering the estimate~\eqref{eq:eigenfunctions} for each eigenpair and plugging all
these estimates in the exponential series.

\begin{cor}\label{cor:trace}
	Under the assumptions of Theorem~\ref{thm:Laplacian},  let $\mE$ be finite with $\mE_\external=\emptyset$. Then, for
	every $(\gamma,\rho)$-sampling set $\omega\subset \fE$, the trace of $\e^{t\Delta_{A,Y}}$ satisfies
	\[
	\Tr(\ee^{t\Delta_{A,Y}})
	\leq
	\frac{1}{12} \Bigl( \frac{48}{\gamma} \Bigr)^5 \sum_{k\in\NN} \ee^{-\lambda_k + \frac{40\rho\sqrt{\lambda_k}}{\log2}
		\log(48/\gamma)}
	\norm{\chi_\omega f_k}{L^2(\Graph)}^2
	,
	\]
	where $(f_k)_{k\in \NN}$ is an orthonormal basis of eigenfunctions of $-\Delta_{A,Y}$ and $(\lambda_k)_{k\in \NN}$ the sequence
	of corresponding eigenvalues.
\end{cor}

\subsection{Observability and controllability of the heat equation}

Combining the (quite general) Theorem~2.8 in \cite{NakicTTV-20} with our Theorem~\ref{thm:Laplacian} we are able to obtain the
following $L^2$-observability estimate for the semigroup $(\ee^{t\Delta_{A,Y}})_{t\geq 0}$ generated by $\Delta_{A,Y}$.

\begin{cor}\label{cor:semigroup-L2-estimate}
	Let $\cE$ and $Y \subset \ell^2(\mE) \oplus \ell^2(\mE_\internal)$ be as in Theorem~\ref{thm:Laplacian}, and let
	$\omega\subset \cE$ be $(\gamma,\rho)$-sampling. Then, for every $T>0$ and $g\in L^2(\cE)$ we have 
	\begin{equation}\label{eq:semigroup-L2-estimate}
		\norm{\ee^{T\Delta_{A,Y}}g}{L^2(\cE)}^2
		\leq 
		C(\omega, T)^2 \int_0^T \norm{\chi_\omega \ee^{t\Delta_{A,Y}}g}{L^2(\cE)}^2 \,\dd t		
	\end{equation}
	with
	\begin{equation}\label{eq:Cobs}
		C(\omega, T)^2 \leq \frac{K_1 \gamma^{-K_2}}{T}\exp\Big(\frac{K_3 \rho^2 \log^2(K_4/\gamma)}{T}\Big),
	\end{equation}
	where $K_j>0$, $j = 1,\ldots, 4$, are universal constants.
\end{cor}

\begin{proof}
	We may rewrite \eqref{eq:eigenfunctions} as
	\[
		\norm{ \PP_{-\Delta_{A,Y}}(\lambda)f }{L^2(\cE)}^2
		\leq
		d_0\ee^{d_1\sqrt{\lambda}} \norm{ \chi_\omega \PP_{-\Delta_{A,Y}}(\lambda)f }{L^2(\cE)}^2
		,\quad
		f \in L^2(\cE)
		,\
		\lambda \geq 0
		,
	\]
	with
	\[
		d_0
		=
		\frac{1}{12}\Bigl(\frac{48}{\gamma}\Bigr)^5
		\quad\text{ and }\quad
		d_1
		=
		\frac{40\rho}{\log 2}\log\Bigl(\frac{48}{\gamma}\Bigr)
		.
	\]
	We therefore deduce from \cite[Theorem~2.8]{NakicTTV-20} that for all $T > 0$ and $g \in L^2(\cE)$ we have
	\[
		\norm{\ee^{T\Delta_{A,Y}}g}{L^2(\cE)}^2
		\leq 
		C(\omega, T)^2 \int_0^T \norm{\chi_\omega \ee^{t\Delta_{A,Y}}g}{L^2(\cE)}^2 \dd t
	\]
	with
	\[
		C(\omega, T)^2
		=
		\frac{C_1 d_0}{T}(2 d_0 +1)^{C_2}\exp\Big(\frac{C_3 d_1^2}{T}\Big)
		,
	\]
	where $C_1, C_2, C_3>0$ are universal constants. Substituting the values for $d_0$ and $d_1$, we see that $C(\omega, T)^2$ can
	be bounded as in \eqref{eq:Cobs}.
\end{proof}

A classical duality argument, see for example \cite{TW09}, shows that the above $L^2$-observability estimate is equivalent to
null-controllability in every positive time $T>0$ of the system 
\begin{equation}\label{eq:parabolic-system}
	\begin{cases}
		\partial_t u(t,\cdot) - \Delta_{A,Y} u(t,\cdot)
		=
		h(t,\cdot)\chi_\omega , & t\in (0,T),\\
		u(0,\cdot)
		=
		u_0 \in L^2(\cE)
		.
	\end{cases}
\end{equation}
This means that for every initial datum $u_0 \in L^2(\cE)$ and for every positive time $T > 0$ there exists a control function
$h$ driving the (mild) solution of the system \eqref{eq:parabolic-system} to zero in time $T$, that is, $u(T,\cdot) = 0$. We
therefore obtain the following result:

\begin{cor}\label{cor:null-controllability}
	Let $\cE$ and $Y \subset \ell^2(\mE) \oplus \ell^2(\mE_\internal)$ be as in Theorem~\ref{thm:Laplacian},
	and let $\omega\subset \cE$ be $(\gamma,\rho)$-sampling. Then \eqref{eq:parabolic-system} is null-controllable in every time
	$T>0$, and it holds
	\begin{multline}
		C_T
		:=
		\sup_{\norm{u_0}{L^2(\cE)}=1} \inf\{\norm{h\chi_\omega}{L^2((0,T)\times\cE)} \mid
		\text{ the solution $u$ of \eqref{eq:parabolic-system} with RHS } h\chi_\omega\\ \text{satisfies $u(T,\cdot)\equiv 0$}\}
		\leq
		C(\omega,T)
		,
	\end{multline}
	where $C(\omega, T)$ is the constant in Corollary~\ref{cor:semigroup-L2-estimate}.
\end{cor}

\begin{rem}\label{rem:nothing-changes-p}
\begin{enumerate}[(1)]

	\item
	By means of the unitary gauge transformation $\gauge_A$ in $L^2(\cE)$ and the corresponding unitary transformation
	$\gaugeBoundary_A$ on the boundary space $\ell^2(\mE) \oplus \ell^2(\mE_\internal)$ that maps the boundary conditions
	accordingly, it is easy to check that $D_{A,Y} = \gauge_A^* D_{0,\gaugeBoundary_A Y} \gauge_A$; we refer the reader to
	Section~\ref{sec:proof-of-main} and the appendix for details on the definition of $\gauge_A$ and $\gaugeBoundary_A$. In turn,
	the magnetic Laplacian $\Delta_{A,Y}$ is unitarily equivalent to the free Laplacian $\Delta_{0,\gaugeBoundary_A Y}$ with
	boundary conditions corresponding to $\gaugeBoundary_A Y \subset \ell^2(\mE) \oplus \ell^2(\mE_\internal)$ via
	$\Delta_{A,Y} = \gauge_A^* \Delta_{0,\gaugeBoundary_AY} \gauge_A$.

	\item
	If $A$ satisfies
	\[
		A_\me \in L^{p_\me}((0,\ell_\me))
		,\quad
		p_\me \in [2,\infty]
		,\
		\me \in \mE
		,
	\]
	such that
	\[
		\sup_{\me \in \mE} \norm{f_\me}{L^{p_\me}((0,\ell_\me))}
		<
		\infty
		,
	\]
	which is, in particular, automatically fulfilled if $A \in L^p(\cE)$ for some $p \in [2,\infty]$, then one can show that
	$W_A(\cE)$ agrees with $W^{1,2}(\cE)$. In this case, the form domains of our corresponding magnetic Laplacians do not depend on
	$A$.

	\item
	The case of part (1) is particularly interesting in the context of a diamagnetic inequality: Just as in the case of magnetic
	Schrödinger operators in the Euclidean setting, the standard diamagnetic inequality together with the abstract criterion in
	\cite[Theorem~2.21]{Ouh05} implies that $(\ee^{t\Delta_{A,Y}})_{t\ge 0}$ is dominated by
	$(\ee^{t\Delta_{0,\gaugeBoundary_A Y}})_{t\geq 0}$ (i.e.,
	$\abs{\ee^{t\Delta_{A,Y}}f} \leq \ee^{t\Delta_{0,\gaugeBoundary_A Y}}\abs{f}$ for all $t\geq 0$ and all
	$f\in L^2(\cE)$) as soon as $(\ee^{t\Delta_{0,\gaugeBoundary_A Y}})_{t\geq 0}$ is positive. By \cite[Theorem~6.85]{Mug14}, this
	is the case if and only if the orthogonal projector of $\bdSp$ onto $\gaugeBoundary_A Y$ is a positivity preserving operator. 
	
	In particular, under the assumptions -- and with the notations -- of Corollary~\ref{cor:semigroup-L2-estimate} we can then
	deduce the estimate
	\begin{equation}\label{eq:semigroup-L2-estimate-2}
		\norm{\ee^{T\Delta_{A,V_A^\ast Y}}f}{L^2(\cE)}^2
		\leq 
		C(\omega, T)^2 \int_0^T \norm{\chi_\omega \ee^{t\Delta_{0,Y}}f}{L^2(\cE)}^2 \dd t
		.
	\end{equation}
	Namely, the semigroup associated to $\Delta_{A,V_A^*Y}$ can be observed at time $T$ by knowledge of the semigroup associated to
	the free Laplacian with boundary condition $Y$ restricted to a $(\gamma,\rho)$-sampling set being a control set for the system
	\eqref{eq:parabolic-system} with $A=0$.
	
\end{enumerate}
\end{rem}

\section{Proof of Theorem~\ref{thm:main}}\label{sec:main_proof}

Recall that $f \in W^{\infty,2}(\cE)$ satisfies a Bernstein-type inequality with respect to $C_B \colon \NN_0 \to [0,\infty)$ such
that \eqref{eq:BernsteinSum} holds for some $\rho > 0$, and let $\omega \subset \cE$ be $(\gamma,\rho)$-sampling.

We begin with reducing the situation of Theorem~\ref{thm:main} to a slightly more specialised one in several aspects:

Firstly, we may assume that each $C_B(m)$ is \emph{strictly positive}. Indeed, by taking $C_B(m) + \eps$, $\eps > 0$, instead of
$C_B(m)$ we may replace $h$ in Theorem~\ref{thm:main} by the sum $h + \eps^{1/2}\exp(10\rho)$ and consider the limit as
$\eps \to 0$ in the final estimate.

Secondly, we observe that \eqref{eq:bernstein} (and, thus, also \eqref{eq:BernsteinSum}) is invariant under subdivision of edges.
This means that we may decompose each $I_\me$ into the adjacent intervals given in terms of the sampling hypothesis on
$\omega_\me$ in $I_\me$, while \eqref{eq:bernstein} and \eqref{eq:BernsteinSum} remain valid. Without loss of generality, we may
therefore assume that
\begin{equation}\label{eq:edgeUpper}
	\lupper
	:=
	\sup_{\me \in \mE} \ell_\me
	\leq
	\rho
	,
\end{equation}
while, at the same time, the $(\gamma,\rho)$-sampling set $\omega$ satisfies
\begin{equation}\label{eq:sampling}
	\abs{\omega_\me}
	\geq
	\gamma\ell_\me
	\quad\text{ for all }\
	\me \in \mE
	.
\end{equation}

With the above reductions, we now classify the edges in $\mE$ into good and bad ones. This is analogous to the approach by
Kovrijkine~\cite{Kov01,Kovrijkine-thesis} in the Euclidean setting and serves the purpose of localizing the Bernstein-type
inequality on good edges.

\begin{defin}\label{def:good-edges}
An edge $\me \in \mE$ is called a \emph{good edge} (with respect to the function $f$ and $C_B \colon \NN_0 \to [0,\infty)$) if
\begin{equation}\label{eq:good-interval}
	\norm{f_\me^{(m)}}{L^2((0,\ell_\me))}^2
	\leq
	2^{m+1} C_B(m)\norm{f_\me}{L^2((0,\ell_\me))}^2
	\quad\text{ for all }\quad
	m \in \NN
	.
\end{equation}
An edge $\me \in \mE$ is called \emph{bad} if it is not good, that is, if there exists $m \in \NN$ such that we have
$\norm{f_\me^{(m)}}{L^2((0,\ell_\me))}^2 > 2^{m+1} C_B(m)\norm{f_\me}{L^2((0,\ell_\me))}^2$. 
\end{defin}

Under the current additional assumption that each $C_B(m)$ is strictly positive, this definition yields
\begin{equation}\label{eq:bad-mass}
	\sum_{\me \colon \me \text{ bad}} \norm{f_\me}{L^2((0,\ell_\me))}^2
	<
	\frac{1}{2} \norm{f}{L^2(\cE)}^2
	.
\end{equation}
Indeed, using the Bernstein-type inequality for $f$, we obtain
\begin{align*}
	\sum_{\me \colon \me\text{ bad}} \norm{f_\me}{L^2((0,\ell_\me))}^2
	&<
	\sum_{\me \colon \me\text{ bad}} \sum_{m\in\NN} \frac{1}{2^{m+1} C_B(m)} \norm{f_\me^{(m)}}{L^2((0,\ell_\me))}^2\\
	&\leq
	\sum_{m\in\NN} \frac{1}{ 2^{m+1} C_B(m)} \norm{f^{(m)}}{L^2(\cE)}^2\\
	&\leq
	\norm{f}{L^2(\cE)}^2 \sum_{m\in\NN} 2^{-m-1}\\
	&=
	\frac{1}{2}\norm{f}{L^2(\cE)}^2
	,
\end{align*}
which proves \eqref{eq:bad-mass}. In particular, good edges exist and we have
\begin{equation}\label{eq:good-mass}
	\norm{f}{L^2(\cE)}^2
	<
	2 \sum_{\me \colon \me \text{ good}} \norm{f_\me}{L^2((0,\ell_\me))}^2.
\end{equation}

\subsection{Analyticity on good edges}

In light of \eqref{eq:good-interval} and \eqref{eq:BernsteinSum}, the following result is an immediate consequence of the one
dimensional case of \cite[Lemma~3.2]{EgiSee21}. For the sake of self-containedness, we reproduce the proof for this particular
case below.

\begin{lemma}\label{lem:analyticity}
	For each good edge $\me\in\mE$, the function $f_{\me}$ is analytic in $(0,\ell_\me)$.
\end{lemma}

\begin{proof}
	Let $\me$ be a good edge, and set $r := \min\{ 1 , 10\rho \} / \sqrt{2}$. Let $y \in (0,\ell_\me)$, and let
	$0 < \varepsilon < r$ be such that the closure of $J := (y - \varepsilon , y + \varepsilon)$ belongs to $(0,\ell_\me)$. By
	Sobolev embedding, there is a constant $c = c(\eps) > 0$ such that
	\[
		\norm{g}{L^\infty(J)}
		\leq
		c \norm{g}{W^{1,2}(J)}
		\quad\text{ for all }\quad
		g \in W^{1,2}(J)
		.
	\]
	We apply this to $g = f_{\me}^{(m)}|_J$ for each $m \in \NN_0$. We then obtain by \eqref{eq:good-interval} that
	\begin{align*}
		\norm{f_{\me}^{(m)}}{L^\infty(J)}^2
		&\leq
		c^2 \norm{f_{\me}^{(m)}}{W^{1,2}(J)}^2
			\leq
			c^2 \norm{f_{\me}^{(m)}}{W^{1,2}((0,\ell_\me))}^2\\
		&=
		c^2 \bigl( \norm{f_{\me}^{(m)}}{L^2((0,\ell_\me))}^2 + \norm{f_{\me}^{(m+1)}}{L^2((0,\ell_\me))}^2 \bigr)\\
		&\leq
		c^2 \norm{f_{\me}}{L^2((0,\ell_\me))}^2 \bigl( 2^{m+1}C_B(m) + 2^{m+2}C_B(m+1) \bigr)
		.
	\end{align*}
	Taking into account that $\sqrt{2}r \leq 1$, this yields
	\begin{align*}
		\norm{f_{\me}^{(m)}}{L^\infty(J)}
		&\leq
		c \norm{f_{\me}}{L^2((0,\ell_\me))} \bigl( 2^{(m+1)/2}C_B(m)^{1/2} + 2^{(m+2)/2}C_B(m+1)^{1/2} \bigr)\\
		&\leq
		\sqrt{2}c \norm{f_{\me}}{L^2((0,\ell_\me))} \frac{(m+1)!}{r^{m+1}} \sum_{k \in \NN_0} C_B(k)^{1/2}
			\frac{(\sqrt{2}r)^k}{k!}
		.
	\end{align*}
	Since $\sqrt{2}r\leq 10\rho$ and in light of \eqref{eq:BernsteinSum}, this is sufficient to conclude that the
	Taylor series of $f_{\me}$ converges absolutely in $J$ and agrees with $f_{\me}$ there. Hence, $f_{\me}$ is analytic in
	$(0,\ell_\me)$.
\end{proof}%

\subsection{The local estimate}
Denote by $D_r \subset \CC$ for $r > 0$ the complex disk of radius $r$ centred at the origin. We use the following local
estimate, which is inspired by Turan's Lemma in \cite{Nazarov-93} and goes back to \cite{Kov01,Kovrijkine-thesis}. It is also
(implicitly) contained in several recent works such as \cite[Theorem~4.5]{GhobberJ-13}, \cite[Section~5]{EgidiV-20},
\cite{WangWZZ-19}, \cite{BeauchardJPS-21}, \cite{MartinPS}. In its current formulation, it is a one dimensional variant of
\cite[Lemma~3.5]{EgiSee21}, and its proof is reproduced below as well.

\begin{lemma}\label{lem:local}
	Let $l > 0$. Moreover, let $g \colon (0,l) \to \CC$ be a non-vanishing function having a bounded analytic
	extension $G \colon (0,l) + D_{4l} \to \CC$. Then, for every measurable set $S \subset (0,l)$ we have
	\[
		\norm{g}{L^2(S)}^2
		\geq
		24\Bigl( \frac{\abs{S}}{48l} \Bigr)^{4\frac{\log M}{\log 2}+1} \norm{g}{L^2((0,l))}^2
	\]
	with
	\[
		M
		:=
		\frac{\sqrt{l}}{\norm{g}{L^2((0,l))}} \cdot \sup_{z \in (0,l)+D_{4l}} \abs{G(z)}
		\geq
		1
		.
	\]
\end{lemma}

The proof of Lemma~\ref{lem:local} relies on the following well-known result by Kovrijkine~\cite{Kov01}.

\begin{lemma}[{\cite[Lemma 1]{Kov01}}]\label{lem:Kovrijkine}
	Let $\phi\colon D_{4+\eps} \to \CC$ for some $\eps > 0$ be an analytic function with $\abs{\phi(0)}\geq 1$. 
	Moreover, let $E \subset [0,1]$ be measurable with positive measure. Then, 
	\[
		\sup_{t\in[0,1]}\abs{\phi(t)}
		\leq
		\Big( \frac{12}{\abs{E}}\Big)^{\frac{2\log M_{\phi}}{\log 2}} \sup_{t\in E} \abs{\phi(t)},
	\]
	where $M_{\phi}= \sup_{z\in D_{4}}\abs{\phi(z)}$.
\end{lemma}

\begin{rem}
	The original formulation of the lemma in \cite{Kov01} requires the function $\phi$ to be analytic in $D_5$ and has the exponent
	$\log M_{\phi} / \log 2$ in the inequality rather than $2\log M_{\phi} / \log 2$. The proof there shows that it suffices for
	$\phi$ to be analytic in a complex disk of radius strictly larger than $4$. However, we were only able to reproduce the result
	with the additional factor $2$ in the exponent. Since this is not essential for our purposes, we then just work with this
	slightly larger exponent here.
\end{rem}

\begin{proof}[Proof of Lemma~\ref{lem:local}]
	Consider the open set
	\[
		W
		=
		\Bigl\{ x\in(0,l) \colon \abs{g(x)} < \Bigl(\frac{\abs{S}}{48 l}\Bigr)^{\frac{2\log M}{\log 2}}
			\frac{\norm{g}{L^2((0,l))}}{\sqrt{l}} \Bigr\}
		.
	\]
	It suffices to prove that $\abs{S}\geq 2\abs{W}$ since we then have $\abs{S\setminus W}\geq \abs{S}/2$ and by the definition of
	$W$ we obtain 
	\begin{align*}
		\norm{g}{L^2(S)}^2
		&\geq
		\norm{g}{L^2(S\setminus W)}^2
			\geq
			\frac{\abs{S}}{2}\Big(\frac{\abs{S}}{48 l}\Big)^{\frac{4\log M}{\log 2}}\frac{\norm{g}{L^2((0,l))}^2}{l}\\
		&=
		24 \Big(\frac{\abs{S}}{48 l}\Big)^{\frac{4\log M}{\log 2}+1}\norm{g}{L^2((0,l))}^2
		.
	\end{align*}

	In order to show that $\abs{S}\geq 2\abs{W}$, we may suppose that $W\neq \emptyset$. We choose $y\in(0,l)$ with
	$\abs{g(y)} \geq \norm{g}{L^2((0,l))} / \sqrt{l}$ and $\xi \in \{ \pm 1 \}$ such that the sub-interval
	$I := \{ y + \xi t \colon t \in [0,l) \} \cap (0,l)$ of $(0,l)$ satisfies
	\[
		\frac{\abs{ I \cap W }}{\abs{I}}
		\geq
		\frac{\abs{W}}{2l}
		.
	\]	
	We clearly have $I = [y,l)$ if $\xi = 1$ and $I = (0,y]$ if $\xi = -1$.
	
	Consider the function 
	\[
		\CC\ni z\mapsto \phi(z)
		=
		\frac{\sqrt{l}}{\norm{g}{L^2((0,l))}} G(y+\abs{I}\xi z).
	\]
	Now, for $\eps$ small enough we have $\abs{I}\xi z\in D_{4l}$ for all $z\in D_{4+\eps}$, so that $\phi$ is analytic in the disk
	$D_{4+\eps}$. Moreover, we have
	\[
		\sup_{t\in[0,1]}\abs{\phi(t)}
		\geq
		\abs{\phi(0)}
		=
		\frac{\sqrt{l}}{\norm{g}{L^2((0,l))}}\abs{g(y)}
		\geq
		1
	\]
	and
	\[
		M_{\phi}
		:=
		\sup_{z\in D_{4}}\abs{\phi(z)}
		\leq
		\frac{\sqrt{l}}{\norm{g}{L^2((0,l))}}\sup_{z\in y+D_{4l}}\abs{G(z)}
		\leq
		M
		.
	\]
	Applying Lemma~\ref{lem:Kovrijkine} to $\phi$ and the set $E=\{t\in[0,1] \colon y+\abs{I}\xi t\in W\}$, and using the fact that
	$G\vert_{(0,l)}=g$ and that $\abs{E} = \abs{I \cap W}/\abs{I} \geq \abs{W}/(2l)$, we obtain
	\begin{align*}
		\sup_{x\in W} \frac{\sqrt{l}}{\norm{g}{L^2((0,l))}}\abs{g(x)}
		&\geq
		\sup_{t\in E} \frac{\sqrt{l}}{\norm{g}{L^2((0,l))}}\abs{g(y+\abs{I}\xi t)}\\
		&=
		\sup_{t\in E}\abs{\phi(t)}
			\geq
			\Big( \frac{\abs{E}}{12}\Big)^{\frac{2\log M_{\phi}}{\log 2}}\sup_{t\in [0,1]}\abs{\phi(t)}
			\geq
			\Big(\frac{\abs{W}}{24 l}\Big)^{\frac{2\log M}{\log 2}}
		.
	\end{align*}
	Combining this with the definition of $W$, we conclude that
	\[
		\sup_{x\in W} \abs{g(x)}
		\leq
		\Big(\frac{\abs{S}}{48 l}\Big)^{\frac{2\log M}{\log 2}}\frac{\norm{g}{L^2((0,l))}}{\sqrt{l}} 
		\leq
		\Big(\frac{\abs{S}}{2\abs{W}}\Big)^{\frac{2\log M}{\log 2}} \sup_{x\in W}\abs{g(x)}
		.
	\]
	Since $g$ is non-vanishing, the latter requires $\abs{S}\geq 2\abs{W}$, which concludes the proof.
\end{proof}%

\subsection{Taylor expansions and conclusion of the proof}
In view of the local estimate in Lemma~\ref{lem:local}, we need to establish that for every good edge $\me$ the function $f_{\me}$
has an analytic extension to a sufficiently large complex neighborhood of $(0,\ell_\me)$. In light of Lemma~\ref{lem:analyticity},
this is done by estimating a suitable Taylor expansion of $f_{\me}$, for which we need an appropriate pointwise estimate for the
derivatives of $f_{\me}$: Analogously to \cite{Kov01,Kovrijkine-thesis}, we claim that in every good edge $\me$ there exists a
point $x_{\me} \in (0,\ell_\me)$ such that
\begin{equation}\label{eq:pointwise}
	\abs{f_{\me}^{(m)}(x_{\me})}
	\leq
	2^{(m+1)}\frac{C_B(m)^{1/2}}{\ell_\me^{1/2}} \norm{f_{\me}}{L^2((0,\ell_\me))}
	\quad\text{ for all }\quad
	m \in \NN_0
	.
\end{equation}
Indeed, assume to the contrary that for every $x \in (0,\ell_\me)$ there exists $m_x \in \NN_0$ such that 
\[
	\abs{f_{\me}^{(m_x)}(x)}^2
	>
	4^{m_x+1}\frac{C_B(m_x)}{\ell_\me} \norm{f_{\me}}{L^2((0,\ell_\me))}^2
	.
\]
It follows
\[
	\frac{1}{\ell_\me} \norm{f_{\me}}{L^2((0,\ell_\me))}^2
	<
	\frac{1}{4^{m_x+1} C_B(m_x)} \abs{f_{\me}^{(m_x)}(x)}^2
	\leq
	\sum_{m\in\NN_0} \frac{1}{4^{m+1} C_B(m)} \abs{f_{\me}^{(m)}(x)}^2
	.
\]
Integrating over $(0,\ell_\me)$ and using the definition of good edges, we obtain 
\begin{align*}
	\norm{f_{\me}}{L^2((0,\ell_\me))}^2
	&<
	\sum_{m\in\NN_0} \frac{1}{4^{m+1} C_B(m)} \norm{f_{\me}^{(m)}}{L^2((0,\ell_\me))}^2\\
	&\leq
	\sum_{m\in\NN_0} \frac{1}{2^{m+1}} \norm{f_{\me}}{L^2((0,\ell_\me))}^2
		=
		\norm{f_{\me}}{L^2((0,\ell_\me))}^2
	,
\end{align*} 
leading to a contradiction. This proves \eqref{eq:pointwise}.

We are finally in position to prove Theorem~\ref{thm:main}.
\begin{proof}[Proof of Theorem~\ref{thm:main}]
	Let $\me\in\mE$ be a good edge, and let $x_{\me} \in (0,\ell_\me)$ be chosen as in \eqref{eq:pointwise}. We first suppose that
	$f_{\me}$ does not vanish. Taking into account \eqref{eq:edgeUpper} and recalling the definition of $h$ in
	\eqref{eq:BernsteinSum}, for every $z \in x_{\me} + D_{5\lupper}$ we then have
	\begin{align*}
		\sum_{m \in \NN_0}\frac{\abs{f_{\me}^{(m)}(x_{\me})}}{m!}\abs{z-x_{\me}}^m
		&\leq
		\sum_{m\in \NN_0}\frac{1}{m!} C_B(m)^{1/2} 2^{(m+1)} (5\lupper)^m \frac{\norm{f_{\me}}{L^2((0,\ell_\me))}}{\ell_\me^{1/2}}\\
		&=
		2 \frac{\norm{f_{\me}}{L^2((0,\ell_\me))}}{\ell_\me^{1/2}} \sum_{m\in\NN_0} C_B(m)^{1/2} \frac{(10\lupper)^m}{m!}\\
		&\leq
		2\frac{\norm{f_{\me}}{L^2((0,\ell_\me))}}{\ell_\me^{1/2}} \sum_{m\in\NN_0} C_B(m)^{1/2} \frac{(10\rho)^m}{m!}\\
		&= 
		2\frac{\norm{f_{\me}}{L^2((0,\ell_\me))}}{\ell_\me^{1/2}}h
			<
			\infty
		.
	\end{align*}
	Hence, the Taylor expansion of $f_{\me}$ around $x_{\me}$ converges in the complex disk $x_{\me} + D_{5\lupper}$. Since
	$(0,\ell_\me) + D_{4\ell_\me} \subset x_{\me} + D_{5\lupper}$ due to $\ell_\me \leq \lupper$ and in light of
	Lemma~\ref{lem:analyticity}, the Taylor expansion of $f_{\me}$ around $x_{\me}$ defines therefore a bounded analytic extension
	$F_{\me} \colon (0,\ell_\me) + D_{4\ell_\me} \to \CC$ of $f_{\me}$ with
	\[
		M_{\me}
		:=
		\frac{\ell_\me^{1/2}}{\norm{f_{\me}}{L^2((0,\ell_\me))}^2} \cdot \sup_{z\in (0,\ell_\me) + D_{4\ell_\me}} \abs{F_{\me}(z)}
		\leq
		2h
		<
		\infty
		.
	\]
	Now, from the local estimate in Lemma~\ref{lem:local} and the sampling hypothesis \eqref{eq:sampling} on $\omega$ we obtain
	\begin{equation}\label{eq:local-estimate}	
		\begin{aligned}
			\norm{f_{\me}}{L^2(\omega_{\me})}^2
			&\geq
			24 \Big(\frac{\abs{\omega_{\me}}}{48\ell_\me}\Big)^{\frac{4\log M_{\me}}{\log 2}+1} \norm{f_{\me}}{L^2((0,\ell_\me))}^2\\
			&\geq
			24 \Big(\frac{\gamma}{48}\Big)^{\frac{4\log (2h)}{\log 2}+1} \norm{f_{\me}}{L^2((0,\ell_\me))}^2\\
			&=
			24 \Big(\frac{\gamma}{48}\Big)^{\frac{4\log h}{\log 2}+5} \norm{f_{\me}}{L^2((0,\ell_\me))}^2
			.
		\end{aligned}
	\end{equation}
	The right-hand side of \eqref{eq:local-estimate} obviously gives a lower bound on $\norm{f_{\me}}{L^2(\omega_{\me})}^2$ also if
	$f_{\me}$ vanishes and, thus, this estimate is valid for \emph{all} good edges $\me\in\mE$. Therefore, taking into account
	\eqref{eq:good-mass}, summing over all good edges finally gives
	\begin{align*}
		\norm{\chi_\omega f}{L^2(\cE)}^2
		&=
		\sum_{\me\in\mE} \norm{f_{\me}}{L^2(\omega_{\me})}^2
			\geq
			\sum_{\me \colon \me \text{ good}} \norm{f_{\me}}{L^2(\omega_{\me})}^2\\
		&\geq
		24 \Big(\frac{\gamma}{48}\Big)^{\frac{4\log h}{\log 2}+5} \sum_{\me \colon \me \text{ good}}
			\norm{f_{\me}}{L^2((0,\ell_\me))}^2\\
		&>
		12\Big(\frac{\gamma}{48}\Big)^{\frac{4\log h}{\log 2}+5} \norm{f}{L^2(\cE)}^2
		,
	\end{align*}
	which concludes the proof of the theorem.
\end{proof}%

\begin{rem}\label{rem:overlap}
	The presented proof heavily relies on the fact that the intervals covering each edge in the definition of sampling sets are
	adjacent, which allowed to decompose each edge and thus reduce to the situation with \eqref{eq:edgeUpper} and
	\eqref{eq:sampling} and work with edges as a whole. If, however, the intervals are allowed to have an essential overlap, as
	suggested in Remark~\ref{rem:extension}, we are forced to work with the separate intervals instead of whole edges, which makes
	the formulations more tedious. Moreover, one has to take care of the overlaps. For instance, if $(J_{\me,k})_k$ denotes the
	family of intervals covering $I_\me$, and if there is no overlap of more than two of these intervals, then
	\[
		\sum_k \norm{g}{L^2(J_{\me,k})}^2
		\leq
		2\norm{g}{L^2(I_\me)}^2
		,\quad
		g \in L^2(I_\me)
		.
	\]
	The latter requires corresponding adaptations to the notion of good and bad intervals (instead of edges) and also to the final
	part of the proof of Theorem~\ref{thm:main}, which results in the slightly different estimate mentioned in
	Remark~\ref{rem:extension}, cf.\ also \cite{EgiSee21}. This, however, is reasonable to consider only if the sampling parameters
	on \emph{all} edges gain from these overlaps.
\end{rem}

\section{Proof of Theorem~\ref{thm:Laplacian}}\label{sec:proof-of-main}

In light of Theorem~\ref{thm:main}, Theorem~\ref{thm:Laplacian} is derived as soon as the functions in the spectral subspaces
$\Ran\PP_{-\Delta_{A,Y}}(\lambda)$, or suitable transformations thereof, satisfy a Bernstein-type inequality in the sense of
Definition~\ref{def:Bernstein} with a suitable $C_B \colon \NN_0 \to [0,\infty)$.

Let $A \in L^1_\loc(\cE)$ and $Y \subset \ell^2(\mE) \oplus \ell^2(\mE_\internal)$ be as in the statement of
Theorem~\ref{thm:main}. Recall (see, e.g., Corollary~\ref{cor:Laplace} below) that $\Delta_{A,Y}$ in \eqref{eq:defMagLap} is the
unique non-positive self-adjoint operator in $L^2(\cE)$ with
\[
	\Dom(\Delta_{A,Y})
	\subset
	\cD_{A,Y}
	:=
	\{ f \in W_A(\cE) \mid (f_\me(0))_{\me\in\mE} \oplus (f_\me(\ell_\me))_{\me\in\mE_\internal} \in Y \}
\]
and
\[
	\langle \ii f' + Af , \ii g' + Ag \rangle_{L^2(\cE)}
	=
	\langle -\Delta_{A,Y}f , g \rangle_{L^2(\cE)}
	,\quad
	f \in \Dom(\Delta_{A,Y})
	,\
	g \in \cD_{A,Y}
	.
\]
Moreover, introducing the notation
\[
	(\ii\partial + A)f
	:=
	\ii f' + Af
	,\quad
	f \in W_A(\cE)
	,
\]
we have
\[
	-\Delta_{A,Y}f
	=
	(\ii\partial + A)^2 f
	,\quad
	f \in \Dom(\Delta_{A,Y})
	.
\]

The following proposition is related to the desired Bernstein-type inequalities, but deals with magnetic derivatives instead
of ordinary ones.

\begin{prop}\label{prop:powerLaplace}
	For all $m \in \NN_0$ and all $f,g \in \Dom^\infty(\Delta_{A,Y}) := \bigcap_{k\in\NN}\Dom(\Delta_{A,Y}^k)$ we have
	\[
		\langle (\ii\partial + A)^m f , (\ii\partial + A)^m g \rangle_{L^2(\cE)}
		=
		\langle (-\Delta_{A,Y})^m f , g \rangle_{L^2(\cE)}
		.
	\]
\end{prop}

\begin{proof}
	Let $f,g \in \Dom^\infty(\Delta_{A,Y})$ and $m \in \NN_0$. If $m$ is even, say $m = 2j$ with $j \in \NN_0$, then
	\[
		\langle (-\Delta_{A,Y})^m f , g \rangle_{L^2(\cE)}
		=
		\langle (-\Delta_{A,Y})^j f , (-\Delta_{A,Y})^j g \rangle_{L^2(\cE)}
		=
		\langle (\ii\partial + A)^m f , (\ii\partial + A)^m g \rangle_{L^2(\cE)}
		.
	\]
	If $m$ is odd, say $m = 2j+1$ with $j \in \NN_0$, then
	$(-\Delta_{A,Y})^j f, (-\Delta_{A,Y})^j g \in \Dom(\Delta_{A,Y}) \subset \cD_{A,Y}$, so that
	\begin{align*}
		\langle (-\Delta_{A,Y})^m f , g \rangle_{L^2(\cE)}
		&=
		\langle (-\Delta_{A,Y})(-\Delta_{A,Y})^j f , (-\Delta_{A,Y})^j g \rangle_{L^2(\cE)}\\
		&=
		\langle (\ii\partial + A)(-\Delta_{A,Y})^j f , (\ii\partial + A)(-\Delta_{A,Y})^j g \rangle_{L^2(\cE)}\\
		&=
		\langle (\ii\partial + A)^m f , (\ii\partial + A)^m g \rangle_{L^2(\cE)}
		.\qedhere
	\end{align*}
\end{proof}%

\begin{rem}\label{rem:powerLaplace}
	In light of \eqref{eq:defMagLap}, Proposition~\ref{prop:powerLaplace}, and its proof, it is easy to see that the quadratic form
	$\fa_{A,Y,m}$ associated to the power $(-\Delta_{A,Y})^m$ is given by
	\[
		\fa_{A,Y,m}[ f , g ]
		=
		\langle (\ii\partial + A)^m f , (\ii\partial + A)^m g \rangle_{L^2(\cE)}
		,\quad
		f,g \in \Dom[\fa_{A,Y,m}]
		,
	\]
	with
	\[
		\Dom[\fa_{A,Y,m}]
		=
		\Bigl\{ f \in W_A(\Graph) \mid (\ii\partial + A)^k f \in W_A(\Graph),\
			\bdEv_+((\ii\partial + A)^k f) \in Y_k ,\ k = 0, \dots, m-1 \Bigr\}
		,
	\]
	where $Y_k = Y$ for $k$ even and $Y_k = \{ (-\alpha)\oplus\beta \mid \alpha \oplus \beta \in Y \}^\perp$ for $k$ odd,
	cf.\ also Remark~\ref{rem:boundary} below.
\end{rem}

In the particular case of $A = 0$, the left-hand side of the equality in Proposition~\ref{prop:powerLaplace} with $g = f$
coincides with $\norm{f^{(m)}}{L^2(\cE)}^2$ and thus provides the means to derive a suitable Bernstein-type equality, see the
proof of the second part of Corollary~\ref{cor:powerLaplace} below. If, however, $A \neq 0$, we first have to use a gauge
transformation to translate the magnetic derivatives into ordinary one. To this end, let us introduce the unitary transformation
\[
	\gauge_A \colon L^2(\cE) \to L^2(\cE)
	,\quad
	f \mapsto
	\biggl( f_\me \cdot \exp\Bigl( -\ii\int_0^\cdot A_\me(s) \,\dd s \Bigr) \biggr)_{\me \in \mE}
	.
\]
It is then easy to see (cf.\ Lemma~\ref{lem:magneticSobolev} below) that $U_A$ maps $W_A(\cE)$ onto
$W^{1,2}(\cE)$ with
\begin{equation}\label{eq:gaugeDerivative}
	(U_Af)'
	=
	U_A((\partial - \ii A)f)
	=
	-\ii U_A((\ii \partial + A)f)
	\in
	L^2(\cE)
	,\quad
	f \in W_A(\cE)
	.
\end{equation}
With this in mind, we derive from Proposition~\ref{prop:powerLaplace} the following result.

\begin{cor}\label{cor:powerLaplace}
	In the situation of Proposition~\ref{prop:powerLaplace}, $U_A$ maps $\Dom^\infty(\Delta_{A,Y})$ into
	$W^{\infty,2}(\cE)$, and we have
	\begin{equation}\label{eq:gaugePower}
		\langle (U_Af)^{(m)} , (U_Ag)^{(m)} g \rangle_{L^2(\cE)}
		=
		\langle (-\Delta_{A,Y})^m f , g \rangle_{L^2(\cE)}
	\end{equation}
	for all $m \in \NN_0$ and all $f,g \in \Dom^\infty(\Delta_{A,Y})$.
	In particular, for every $f \in \Ran \PP_{-\Delta_{A,Y}}(\lambda)$ the function $U_Af \in W^{\infty,2}(\cE)$
	satisfies a Bernstein-type inequality with respect to $C_B \colon \NN_0 \to [0,\infty)$, $C_B(m) = \lambda^m$.
\end{cor}

\begin{proof}
	For $f \in \Dom^\infty(\Delta_{A,Y})$ and $m \in \NN_0$ we clearly have $(\ii\partial + A)^m f \in W_A(\cE)$.
	Iterating \eqref{eq:gaugeDerivative}, we thus obtain
	\[
		(U_Af)^{(m)}
		=
		U_A((\partial - \ii A)^m f)
		=
		(-\ii)^m U_A((\ii\partial + A)^m f)
		\in
		L^2(\cE)
		,\quad
		f \in \Dom^\infty(\Delta_{A,Y}),\ m \in \NN_0
		.
	\]
	In particular, $U_A$ indeed maps $\Dom^\infty(\Delta_{A,Y})$ into $W^{\infty,2}(\cE)$, and the assertion of
	Proposition~\ref{prop:powerLaplace} simply rewrites as \eqref{eq:gaugePower} since $U_A$ is unitary. The remaining statement
	then follows from \eqref{eq:gaugePower} with $g = f$ by functional calculus,
	\[
		\norm{ (U_Af)^{(m)} }{L^2(\cE)}^2
		=
		\langle (-\Delta_{A,Y})^m f , f \rangle_{L^2(\cE)}
		\leq
		\lambda^m \langle f , f \rangle_{L^2(\cE)}
		=
		\lambda^m \norm{ U_Af }{L^2(\cE)}^2
		,
	\]
	which completes the proof.
\end{proof}%

Before we finally turn to the proof of Theorem~\ref{thm:Laplacian}, let us note that for every
$f \in \Ran\PP_{-\Delta_{A,Y}}(\lambda)$ we have
\begin{equation}\label{eq:magneticDual}
	(\ii\partial + A)f
	\in
	\Ran\PP_{-\Delta_{A,Y_-}}(\lambda)
	,
\end{equation}
where
\[
	Y_-
	=
	\{ (-\alpha) \oplus \beta \mid \alpha \oplus \beta \in \ell^2(\mE) \oplus \ell^2(\mE_\internal) \}^\perp
	.
\]
This is a consequence of the fact that the operators $-\Delta_{A,Y}$ and $-\Delta_{A,Y_-}$ are in some sense dual to one
another, see Corollary~\ref{cor:Laplace} below.

\begin{proof}[Proof of Theorem~\ref{thm:Laplacian}]
	By Corollary~\ref{cor:powerLaplace}, for $f \in \Ran\PP_{-\Delta_{A,Y}}(\lambda)$ the function $U_Af$ satisfies a Bernstein-type
	inequality with respect to $C_B \colon \NN_0 \to [0,\infty)$ with $C_B(m) = \lambda^m$. The corresponding quantity $h$
	in~\eqref{eq:BernsteinSum} can be written as $h = \exp(10\rho\sqrt{\lambda})$. The claim for $f$ itself is then an immediate
	consequence of Theorem~\ref{thm:main} applied to $U_Af$ and the fact that
	\[
	  \norm{U_Af}{L^2(\cE)}
		=
		\norm{f}{L^2(\cE)}
		\quad\text{ and }\quad
		\norm{\chi_\omega U_Af}{L^2(\cE)}
		=
		\norm{\chi_\omega f}{L^2(\cE)}
		.
	\]
	In light of \eqref{eq:magneticDual}, the statement for $(\ii\partial + A)f$ follows in the same way with $\Delta_{A,Y}$ replaced
	by $\Delta_{A,Y_-}$.
\end{proof}

\section{Metric graphs}\label{sec:special-case-graph}

While we have seen that our main results turn out to hold for a rather broad class of operators acting on functions supported on
$\cE$, as long as they admit a natural factorisation, we now present how our theory can be specialised if $\cE$ is identified
with the edge set \emph{of a metric graph}; we refer to \cite{Mug19} for a precise introduction to this class of metric measure
spaces. For convenience, we abbreviate
\[
	\bdSp
	=
	\ell^2(\mE) \oplus \ell^2(\mE_\internal)
	,
\]
and equip $\cH$ with the natural inner product.

We start by discussing how different choices of $Y$ on respective metric graphs lead to certain realisations of the
(magnetic) Laplacian.

\begin{exa}\label{exa:vert-cond}
	For simplicity, we consider the case $A = 0$, although, in view of Remark~\ref{rem:nothing-changes-p}\,(2), nothing substantial
	would change by allowing for, say, $0\ne A\in L^p(\Graph)$ with some $p\in [2,\infty]$.
	
	\begin{enumerate}[(1)]

		\item
		The choice $Y = \{ 0 \}$ leads to the (decoupled) Dirichlet Laplacian, whereas the choice $Y = \bdSp$ (that is,
		$Y^\perp = \{ 0 \}$) corresponds to the (decoupled) Neumann Laplacian.
	
		\item
		Suppose, in addition, that the metric graph $\Graph$ is locally finite, i.e., that for each vertex $\mv \in \mV$
		the set $\mE_\mv$ of edges incident in $\mv$ (i.e., of edges one of whose endpoints is identified with $\mv$) is finite.
		In this case, the most common realisation of the free Laplacian, the \emph{standard Laplacian}, is obtained by imposing
		continuity and Kirchhoff-type vertex conditions. More precisely, for each vertex $\mv$ we define vectors
		$\iota^+_{\mv} = (\iota^+_{\mv,\me})_{\me \in \mE} \in \ell^2(\mE)$ and
		$\iota^-_{\mv} = (\iota^-_{\mv,\me})_{\me\in\mE_\internal} \in \ell^2(\mE_\internal)$ by
		\[
			\iota^+_{\mv,\me}
			=
			\begin{cases}
				1 ,& \mv \text{ is initial endpoint for }\me,\\
				0 ,& \text{otherwise},
			\end{cases}
			\qquad
			(\me \in \mE)
		\]
		and
		\[
			\iota^-_{\mv,\me}
			=
			\begin{cases}
				1 ,& \mv \text{ is terminal endpoint for }\me,\\
				0 ,& \text{otherwise},
			\end{cases}
			\qquad
			(\me \in \mE_\internal)
			.
		\]
		The standard Laplacian is then generated by choosing $Y$ as the closure of the linear span of the vectors
		$y_\mv := \iota^+_\mv \oplus \iota^-_\mv \in \bdSp$, $\mv \in \mV$, that is,
		\begin{equation}\label{eq:Y-egisee}
			Y
			=
			\overline{\Span_\CC\{ y_\mv \in \bdSp \colon \mv \in \mV \}}
			.
		\end{equation}
		Note that each $y_\mv$ indeed belongs to $\bdSp$ since the graph is assumed to be locally finite and, therefore, $y_\mv$ has
		only finitely many non-zero entries. It is then not hard to see that every $f \in \Dom(\Delta_{0,Y})$ is continuous across the
		vertices and satisfies the Kirchhoff-type condition
		\[
			\sum_{\me \in \mE} 	\iota^-_{\mv,\me} f_\me'(\ell_\me) - \sum_{\me \in \mE_\mv} \iota^+_{\mv,\me} f_\me'(0)
			=
			0
			,\quad
			\mv \in \mV
			,
		\]
		i.e., at each vertex the sum of incoming flows equals the sum of outgoing flows.
	
		\item
		In the situation of the preceding item, another prominent choice for $Y$ arises, for instance, by taking
		\[
			Y
			=
			\Span_\CC\{ (-\iota^+_\mv) \oplus \iota^-_\mv \in \bdSp \colon \mv \in \mV \}^\perp
		\]
		with $\iota^\pm_\mv$ as above, cf.\ Remark~\ref{rem:boundary} below. These are sometimes called
		\emph{anti-Kirchhoff conditions} and correspond to $\delta'$-con\-di\-tions with vanishing coupling strength. It is worth to
		note that they are dual to Kirchhoff boundary conditions in the sense of Corollary~\ref{cor:Laplace}\,(b) below.

		\item
		Further covered realisations are the ones that involve appropriate weights in the
		continuity and Kirchhoff-type vertex conditions, and whose role in the theory of positivity preserving semigroups and strict
		positivity of ground states have been discussed in \cite[Theorem~6.85]{Mug14} and \cite{Kur19}. However, $\delta$- or
		$\delta'$-couplings with non-trivial coupling parameters do not fit into the above framework.

	\end{enumerate}
\end{exa}

We now show how the estimates \eqref{eq:eigenfunctions} and \eqref{eq:eigenfunctionsDer} can become uniform in the geometric
parameters of a certain class of metric graphs: To begin with, observe that if a metric graph $\Graph$ is connected and has
finitely many edges, all of them of finite length, then the spectrum of $\Delta_{A,Y}$ on $L^2(\Graph)$ consists of a sequence of
real, nonnegative eigenvalues that accumulate at $+\infty$, cf.\ Lemma~\ref{lem:magneticSobolev} and Corollary~\ref{cor:Laplace}
below. We denote these eigenvalues in non-decreasing order and counting multiplicities by
\[
	\lambda_1
	\leq
	\lambda_2
	\leq
	\lambda_3
	\leq
	\ldots
	.
\]
If we now consider the standard Laplacian, i.e., the realisation discussed in Example~\ref{exa:vert-cond}\,(2) (in particular, we
assume $A=0$), and if $\Graph$ is not a cycle, then $\lambda_1=0$ is a simple eigenvalue and the two-sided (sharp) estimate
\[
	\frac{k^2 \pi^2}{4 |\Graph|^2}
	\leq
	\lambda_k
	\leq
	\left(k-1+\frac{3}{2}\beta+ \frac{\abs{N}}{2}\right)^2 \frac{\pi^2}{|\Graph|^2}
	,\quad
	k = 2,3,\ldots,
\]
on the $k$-th eigenvalue is known to hold, see \cite[Theorem~1]{Fri05}, \cite[Theorem~4.9]{BerKenKur17}, and
\cite[Theorem~2]{KurSer18}. Here, $\abs{N}$ denotes the number of vertices of degree $1$, on which -- by definition -- plain
Neumann conditions are imposed, $\beta := \abs{E} - \abs{V} + 1$ is the \emph{Betti number of $\Graph$}, and
$\abs{\Graph} := \sum_{\me \in \mE} \ell_\me$ denotes the \emph{total length} of $\Graph$. An alternative upper bound involving
the diameter $D$ of the graph, i.e.\ the supremum of all distances between any two points in $\Graph$, has recently been obtained
in \cite[Theorem~5.2]{DufKenMug22}, namely
\[
	\lambda_k
	\leq
	(k+\beta-1)^2 \frac{\pi^2}{D^2}
	,\quad
	k = 2,3,\ldots.
\]
As a consequence, Theorem~\ref{thm:Laplacian} implies estimates for linear combinations of eigenfunctions corresponding to the
$k$ lowest eigenvalues that are uniform in the total length $\abs{\Graph}$ of $\Graph$, the Betti number $\beta$, and the number
$\abs{N}$ of vertices of degree $1$ (resp.\ the diameter of the graph).

\begin{cor}\label{cor:appl-metr-gr}
	Let $\Graph$ be a finite metric graph of finite total length. Let $f \neq 0$ be a finite linear combination of eigenfunctions
	corresponding to the $k$ lowest eigenvalues (counting multiplicities) of the standard Laplacian $\Delta_\Graph$ on $\Graph$,
	as in Example~\ref{exa:vert-cond}\,(2). Then, for every $(\gamma,\rho)$-sampling set $\omega \subset \fE$ we have
	\[
		\norm{\chi_\omega f}{L^2(\Graph)}^2
		>
		C \norm{f}{L^2(\Graph)}^2
	\]
	and
	\[
		\norm{\chi_\omega f'}{L^2(\Graph)}^2
		>
		C \norm{f'}{L^2(\Graph)}^2
	\]
	with a constant $C > 0$ satisfying
	\[
		12\Bigl( \frac{\gamma}{48} \Bigr)^{\frac{40\rho}{\log 2}\bigl(k-1+\frac{3}{2}\beta + \frac{\abs{N}}{2}\bigr)
			\frac{\pi}{\abs{\Graph}}+5}
		\leq 
		C
		\leq 
		12\Bigl( \frac{\gamma}{48} \Bigr)^{\frac{20\rho k \pi}{\abs{\Graph}\log 2}+5}
		.
	\]
	Furthermore, for metric graphs of diameter $D$ the constant $C$ can be chosen such that we have the lower bound
	\[
		C
		\geq
		12\Bigl( \frac{\gamma}{48} \Bigr)^{\frac{40\rho(k+\beta-1)\pi}{D\log 2}+5}
		.
	\]
\end{cor}

We point out that also the results of Section~\ref{sec:appllications-semig} can be formulated for metric graph
simply by replacing $\cE$ with $\Graph$. In particular, if $A=0$ and the two-sided eigenvalue estimates from above
are taken into account, the statement of Corollary~\ref{cor:trace} can be formulated only in terms of geometric properties of
the graph.

\begin{rem}\label{rmk:unique-continuation}
	As mentioned in the introduction, on metric graphs the unique continuation principle fails to hold in general. This severely
	reduces the possibility of extending our main results to general subsets $\omega$ of metric graphs. One may, for instance,
	na\"ively wonder whether the sampling property of $\omega$ in the sense of Definition~\ref{def:gammarhosampling} may be relaxed
	by merely assuming that a \emph{subgraph} $\omega\subset \Graph$ satisfies 
	\begin{equation}\label{eq:gamma-thickness-3}
		\abs{\omega}
		\geq
		\gamma |\Graph|
		,
	\end{equation}
	that is, assuming that a fraction of $\Graph$ -- but not necessarily of \emph{each} edge -- belongs to $\omega$.
	However, Theorem~\ref{thm:Laplacian} generally fails to hold under this weaker assumption, and a counterexample is given by an
	eigenfunction supported on the loop of a lasso graph $\Graph$, see Figure~\ref{fig:wagen}.
	\begin{figure}[H]
		\begin{tikzpicture}
			\draw[very thick, blue] (.2,0) -- (0,1) -- (1.1,1) -- (1.3,0) -- cycle;
			\draw[fill = blue] (0,1) circle (1mm);
			\draw[fill = blue] (1.1,1) circle (1mm);
			\draw[fill = blue] (1.3,0) circle (1mm);
			\draw[fill = blue] (.2,0) circle (1mm);

			\draw[fill = blue] (-.75,1) circle (1mm);
			\draw[fill = blue] (-.5,1.25) circle (1mm);
			\draw[fill = blue] (-.25,1.1) circle (1mm);
			\draw[very thick, blue] (-.75,1) -- (-.5,1.25) -- (-.25,1.1) -- (0,1);

			\begin{scope}
				\draw[thick, red, rounded corners = 2mm]
				(-0.1,1) -- (.275,1.25) -- (0.875,.75) -- 
				(1.2,1.1) -- (1.3,.72)-- (1.05,.245) -- 
				(1.4,0) -- (1.075,-.25) -- (.475,.25) --	
				(.15,-.1) -- (-.05,.245)  -- (.3,.75) -- 
				cycle ;
			\end{scope}
		\end{tikzpicture}
		\caption{A metric graph where the unique continuation principle is not satisfied; the second eigenfunction of the Laplacian
		with standard vertex conditions is depicted in red.}\label{fig:wagen}
	\end{figure}
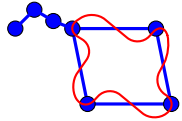
	
	On the other hand, it seems plausible that the fact that $\omega$ must intersect all edges is an artefact of our method, which
	is based on complex analytical tools that seem to be difficult to extend from individual edges to network-like ramified spaces, 
	and that this condition might be relaxed for at least some classes of metric graphs. Indeed, at least in the case of metric
	graphs with pairwise incommensurable edge lengths it is known that the eigenfunctions cannot vanish identically on any open
	subgraph, see \cite[Section~3.4]{BerKuc13}.
\end{rem}

We conclude this section by applying Theorem~\ref{thm:main} to the so-called \emph{torsion function}, which we have
already encountered in the introduction.

\begin{exa}\label{exa:torsion}
	Given a compact metric graph, its \emph{torsion function} $u$ for an arbitrary strictly positive, self-adjoint realisation of
	$\Delta$ with positive resolvent is the (unique) solution of 
	\begin{equation}\label{eq:polya0}
		-\Delta u(x) = 1
		,\quad
		x\in \Graph
		,
	\end{equation}
	see~\cite{MugPlu23}. If the vertex conditions of such a realisation are either of Dirichlet or of standard (continuity and
	Kirchhoff) type -- on some set $\emptyset \neq \mVD \subset \mV$ and on $\mV\setminus\mVD$, respectively -- then it is known
	(see the proof of \cite[Proposition~5.1]{MugPlu23}) that
	\[
		\frac{T(\Graph;\mVD)}{\abs{\Graph}}
		<
		\frac{\norm{u}{L^2(\Graph)}^2}{\norm{u'}{L^2(\Graph)}^2}
		,
	\]
	where $\abs{\Graph}$ again denotes the total length of $\Graph$ and $T(\Graph;\mVD) := \norm{u}{L^1(\Graph)} < \infty$ is the
	so-called \emph{torsional rigidity} of $\Graph$. Furthermore,
	\[
		\norm{u''}{L^2(\Graph)}^2
		=
		\norm{\mathbf{1}}{L^2(\Graph)}^2
		=
		\abs{\Graph}
	\]
	and
	\[
		\norm{u^{(m)}}{L^2(\Graph)}^2
		=
		\norm{(\mathbf{1})^{(m-2)}}{L^2(\Graph)}^2
		=
		0
		\quad\text{ for all }\
		m\ge 3
		,
	\]
	so that $u$ is just a polynomial of degree $2$ on each edge. This shows that the torsion function $u$ of a metric graph
	satisfies a Bernstein-type inequality with respect to $C_B \colon \NN_0 \to [0,\infty)$ with
	\begin{equation}\label{eq:CB-expla}
		C_B(0) = 1
		,\quad
		C_B(1) = \frac{\abs{\Graph}}{T(\Graph;\mVD)}
		,\quad
		C_B(2) = \frac{\abs{\Graph}}{\norm{u}{L^2(\Graph)}^2} < \frac{\abs{\Graph}^2}{T(\Graph;\mVD)^2}
		,
	\end{equation}
	where the last inequality follows from Cauchy--Schwarz, and $C_B(m)=0$ for all $m\geq 3$. Therefore, Theorem~\ref{thm:main}
	applies and leads for every $(\gamma,\rho)$-sampling set $\omega \subset \cE$ to the estimate
	\begin{equation}\label{eq:ls-torsion-general}
		\norm{u\chi_\omega}{L^2(\Graph)}^2
		>
		12\Bigl( \frac{\gamma}{48} \Bigr)^{\frac{4\log h}{\log 2}+5} \norm{u}{L^2(\Graph)}^2
	\end{equation}
	with
	\begin{equation}\label{eq:h-expla}
		h
		=
		\sqrt{C_B(0)} + \sqrt{C_B(1)}10\rho + \sqrt{C_B(2)}\frac{(10\rho)^2}{2}
		.
	\end{equation}
	Here, the quantity $h$ still depends on the function $u$ itself but can be bounded by something that only depends on geometric
	properties of the graph. As a first step, we observe that the right-hand side of \eqref{eq:ls-torsion-general} is non-increasing
	with respect to $h$. In view of \eqref{eq:CB-expla}, we may therefore replace $h$ in \eqref{eq:ls-torsion-general} by
	\[
		h'
		=
		1 + \frac{10\rho\sqrt{\abs{\Graph}}}{\sqrt{T(\Graph;\mVD)}} + \frac{50\rho^2\abs{\Graph}}{T(\Graph;\mVD)}
	\]
	and look for further upper bounds on the quotient $\abs{\Graph} / T(\Graph;\mVD)$. Now, the torsional rigidity $T(\Graph;\mVD)$
	is a quantity that only depends on $\Graph$ and $\mVD$ and can be computed explicitly, by
	\cite[Proposition~3.1 and Theorem~3.9]{MugPlu23}. If the graph has $\abs{\mE}$ many edges and Dirichlet conditions are imposed
	on at least one vertex (i.e.\ $\mVD \neq \emptyset$), the upper bound
	\[
		\frac{\abs{\Graph}}{T(\Graph;\mVD)}
		\leq
		\frac{12|\mE|^3}{\abs{\Graph}^2}
	\]
	is known to hold, see \cite[(4.6) and the improved estimate in Proposition~4.8]{MugPlu23}. In the particular case where the
	metric graph consists of only one interval of length $\abs{\Graph} = \ell$, we even have by \cite[Example~2.4]{MugPlu23} that
	\[
		T(\Graph;\mVD)
		=
		\frac{\ell^3}{3}
		\quad\text{ or }\quad
		T(\Graph;\mVD)
		=
		\frac{\ell^3}{12}
		,
	\]
	depending on whether $\mVD$ consists of one or two vertices, respectively. If we further choose
	$\omega = (\frac{\ell}{4}, \frac{3\ell}{4})$ and recall from Example~\ref{exa:basic}\,(3) that this set is
	$(\gamma=\frac{1}{2},\rho=\frac{\ell}{2})$-sampling, we find that
	\begin{equation}\label{eq:BernsteinSum-torsion}
		\norm{u\chi_\omega}{L^2(\Graph)}^2
		>
		12\Bigl( \frac{1}{96} \Bigr)^{\frac{4\log h'}{\log 2}+5} \norm{u}{L^2(\Graph)}^2,
	\end{equation}
	with
	\[
		h'
		=
		1 + 5\sqrt{3} + \frac{75}{2} < 48
		\quad\text{ resp.\ }\quad 
		h'
		=
		1 + 10\sqrt{3} + 150<169
		.
	\]

	Of course, the general estimate \eqref{eq:ls-torsion-general} can be further combined with (the graph counterpart of)
	\eqref{eq:filmay}, thus yielding a global estimate on any eigenfunction of the standard Laplacian with Dirichlet conditions
	imposed on at least one vertex in terms of the norm of the torsion function on a small control set.
\end{exa}

\appendix

\section{Factorisation of vector-valued magnetic Laplacians with possibly non-separated boundary conditions}
\label{sec:elliptic-metric}

In this appendix, we review the class of operators considered in the main part of the article, that is, the self-adjoint
realisations $\Delta_{A,Y}$ of the magnetic Laplacian on the edge set $\cE$ associated to real-valued $A \in L_\loc^1(\cE)$ and
closed subspaces $Y$ of $\ell^2(\mE) \oplus \ell^2(\mE_\internal)$. It is worth to note once more that a similar study for
Schrödinger operators with magnetic potentials on Euclidean domains can be found in \cite{HundertSimon-03}.
Assumption~\ref{assum-main} is imposed throughout the following.

We write 
\[
	\bdSp
	=
	\ell^2(\mE) \oplus \ell^2(\mE_\internal)
\]
for the \emph{boundary space} and equip it with the natural inner product $\langle \cdot , \cdot \rangle_\bdSp$ and corresponding
norm $\norm{\cdot}{\bdSp}$. We also denote
\[
	\preboundary(\cE)
	:=
	\bigl\{ f \colon \cE \to \CC \mid f_\me \in C(I_\me)\ \forall\me\in\mE,\ (f_\me(0))_{\me\in\mE} \oplus
	(f_\me(\ell_\me))_{\me\in\mE_\internal} \in \bdSp \bigr\}
	,
\]
and introduce the \emph{boundary evaluation maps} $\bdEv_\pm \colon \preboundary(\cE) \to \bdSp$ by
\[
	\bdEv_\pm(f)
	=
	(\pm f_\me(0))_{\me\in\mE} \oplus (f_\me(\ell_\me))_{\me\in\mE_\internal}
	.
\]

We rely in the following on Sobolev embeddings of first order on an interval. While we content ourselves with a qualitative
result in the $L^1$-case -- namely that $W^{1,1}((a,b))$ for $a,b \in \RR$ with $a < b$ is continuously embedded in $C([a,b])$ --
we use for the $L^2$-case the optimal result in this context obtained in \cite{Mar83} and obtain the following somewhat sharper
version of~\cite[Lemma~2.2.3]{Pos12}; a different but comparable result for metric graphs that are not of semi-bounded geometry
has been obtained in~\cite[Lemma~3.2 and Remark~3.3]{KosMugNic22}.

\begin{lemma}\label{lem:boundary}
	The space $W^{1,2}(\cE) = \bigoplus_{\me\in\mE} W^{1,2}((0,\ell_\me))$ is contained in $\preboundary(\cE)$, and the restrictions
	of the operators $\bdEv_\pm$ to $W^{1,2}(\cE)$ are bounded and have bounded right inverses. More precisely, we have
	\begin{equation}\label{eq:marti-appl}
		\norm{\bdEv_\pm(f)}{\bdSp}^2
		\leq
		2\coth(\llower)\norm{f}{W^{1,2}(\cE)}^2
		\quad\text{ for all }\
		f \in W^{1,2}(\cE)
		,
	\end{equation}
	and for every $\nu \in \bdSp$ there are $f^\pm \in W^{1,2}(\cE)$ with $\bdEv_\pm(f^\pm) = \nu$ and
	\[
	\norm{f^\pm}{W^{1,2}(\cE)}^2
	=
	\Bigl( \frac{\llower}{6} + \frac{2}{\llower} \Bigr) \norm{\nu}{\bdSp}^2
	.
	\]
\end{lemma}

\begin{proof}
	Let $f \in W^{1,2}(\cE)$. By the Sobolev embedding from \cite{Mar83}, for every $\me \in \mE$ we then have
	$f_\me \in C(I_\me)$ and
	\[
	\abs{f_\me(0)}^2
	\leq
	\coth(\llower) \norm{f_\me}{W^{1,2}((0,\llower))}^2
	\leq
	\coth(\llower) \norm{f_\me}{W^{1,2}((0,\ell_\me))}^2
	,
	\]
	and for $\me \in \mE_\internal$ also
	\[
	\abs{f_\me(\ell_\me)}^2
	\leq
	\coth(\llower) \norm{f_\me}{W^{1,2}((\ell_\me-\llower,\ell_\me))}^2
	\leq
	\coth(\llower) \norm{f_\me}{W^{1,2}((0,\ell_\me))}^2
	.
	\]
	Summing over $\me \in \mE$, respectively $\me \in \mE_\internal$, proves \eqref{eq:marti-appl}. In particular,
	$W^{1,2}(\cE)$ is indeed contained in $\preboundary(\cE)$, and the restrictions of $\bdEv_\pm$ to $W^{1,2}(\cE)$ are bounded.
	
	Conversely, let $\nu = \alpha \oplus \beta \in \bdSp$. We define $f^\pm \in L^2(\cE)$ as follows: For $\me \in \mE_\internal$
	let $f^\pm_\me \in L^2((0,\ell_\me))$ be given by
	\[
	f^\pm_\me(x)
	:=
	\begin{cases}
		\pm\alpha_\me(1 - 2x/\llower) ,& 0 \leq x \leq \llower/2,\\
		0 ,& \llower/2 < x < \ell_\me - \llower/2,\\
		\beta_\me(1 + 2(x-\ell_\me)/\llower) ,& \ell_\me - \llower/2 \leq x \leq \ell_\me.
	\end{cases}
	\]
	Then clearly $f^\pm_\me$ belongs to $W^{1,2}((0,\ell_\me))$ with $f^\pm_\me(0) = \pm\alpha_\me$ and
	$f^\pm_\me(\ell_\me) = \beta_\me$. Moreover,
	\[
	\norm{f^\pm_\me}{L^2((0,\ell_\me))}^2
	=
	\frac{\llower}{6}( \abs{\alpha_\me}^2 + \abs{\beta_\me^2} )
	\quad\text{ and }\quad
	\norm{(f^\pm_\me)'}{L^2((0,\ell_\me))}^2
	=
	\frac{2}{\llower} ( \abs{\alpha_\me}^2 + \abs{\beta_\me}^2 )
	.
	\]
	For $\me \in \mE_\external = \mE \setminus \mE_\internal$, let $f^\pm_\me \in L^2((0,\ell_\me))$ be given by
	\[
	f^\pm_\me(x)
	:=
	\begin{cases}
		\pm\alpha_\me(1 - 2x/\llower) ,& 0 \leq x \leq \llower/2,\\
		0 ,& \llower/2 < x.
	\end{cases}
	\]
	Then again $f^\pm_\me$ belongs to $W^{1,2}((0,\ell_\me))$ with $f^\pm_\me(0) = \pm\alpha_\me$ and satisfies
	\[
	\norm{f^\pm_\me}{L^2((0,\ell_\me))}^2
	=
	\frac{\llower}{6} \abs{\alpha_\me}^2
	\quad\text{ and }\quad
	\norm{(f^\pm_\me)'}{L^2((0,\ell_\me))}^2
	=
	\frac{2}{\llower} \abs{\alpha_\me}^2
	.
	\]
	Summing over $\me \in \mE = \mE_\internal \cup \mE_\external$ finally shows that indeed $f^\pm \in W^{1,2}(\cE)$
	with $\bdEv_\pm(f^\pm) = \nu$ and
	\[
	\norm{ f^\pm }{W^{1,2}(\cE)}^2
	=
	\norm{ f^\pm }{L^2(\cE)}^2 + \norm{ (f^\pm)' }{L^2(\cE)}^2
	=
	\Bigl( \frac{\llower}{6} + \frac{2}{\llower} \Bigr) \norm{\nu}{\bdSp}^2
	.
	\]
	This proves the remaining part of the statement and, thus, concludes the proof.
\end{proof}

Given a real-valued $A \in L_\loc^1(\cE)$, recall that
\[
W_A(\cE)
=
\{ f \in L^2(\cE) \cap W_\loc^{1,1}(\cE) \colon (\ii\partial + A)f \in L^2(\cE) \}
\]
with $(\ii\partial + A)f = \ii f' + Af$ and that the associated unitary gauge transformation $\gauge_A$ in $L^2(\cE)$ is given by
\[
\gauge_A f
=
\biggl( f_\me \cdot \exp\Bigl( -\ii\int_0^\cdot A_\me(s) \,\dd s \Bigr) \biggr)_{\me \in \mE}
,\quad
f \in L^2(\cE)
.
\]
We also introduce the unitary transformation $\gaugeBoundary_A \colon \cH \to \cH$ in the boundary space by
\[
V_A(\alpha \oplus \beta)
=
\alpha \oplus \biggl(\beta_\me \cdot \exp\Bigl(-\ii\int_0^{\ell_\me} A_\me(s) \,\dd s\Bigr)\biggr)_{\me \in \mE_\internal}
.
\]
Note that $\gauge_A$ and $\gaugeBoundary_A$ are indeed well-defined since $A_\me \in L_\loc^1(I_\me)$ for all $\me \in \mE$, that
$\gauge_A$ maps $L^2(\cE) \cap \preboundary(\cE)$ into $L^2(\cE) \cap \preboundary(\cE)$, and that
\begin{equation}\label{eq:boundaryUnitary}
	\gaugeBoundary_A \circ \Psi_\pm
	=
	\Psi_\pm \circ \gauge_A
	\quad\text{ on }\quad
	L^2(\cE) \cap \preboundary(\cE)
	.
\end{equation}

We equip $W_A(\cE)$ with the inner product
\[
\langle f , g \rangle_{W_A(\cE)}
:=
\langle f , g \rangle_{L^2(\cE)} + \langle (\ii\partial + A)f , (\ii\partial + A)g \rangle_{L^2(\cE)}	
\]
and obtain the following relation between $W_A(\cE)$ and $W^{1,2}(\cE)$, which also serves as an alternative characterisation of
the space $W_A(\cE)$.

\begin{lemma}\label{lem:magneticSobolev}
	The gauge transformation $\gauge_A$ is a unitary transformation from $W_A(\cE)$ to $W^{1,2}(\cE)$ and every $f \in W_A(\Graph)$
	satisfies
	\begin{equation}\label{eq:magneticDerivative}
		(\gauge_A f)'
		=
		-\ii\gauge_A((\ii\partial + A)f)
		.
	\end{equation}
	In particular, $W_A(\cE)$ is a Hilbert space with respect to the inner product $\langle \cdot , \cdot \rangle_{W_A(\cE)}$,
	$W_A(\cE)$ is contained in $\preboundary(\cE)$ as a set, and the restrictions of $\bdEv_\pm$ to $W_A(\cE)$ are surjective.
	
	Moreover, the embedding of $W_A(\cE)$ in $L^2(\cE)$ is Hilbert--Schmidt if, in addition, $\mE$ is finite and each edge has
	finite length, that is, $\mE = \mE_\internal$.
\end{lemma}

\begin{proof}
	We show that $\gauge_A W_A(\cE) = W^{1,2}(\cE)$ as sets and that \eqref{eq:magneticDerivative} holds.
	This then implies that
	\[
	\norm{ \gauge_A f }{W^{1,2}(\cE)}^2
	=
	\norm{ \gauge_A f }{L^2(\cE)}^2 + \norm{ (\gauge_A f)' }{L^2(\cE)}^2
	=
	\norm{ \gauge_A f }{L^2(\cE)}^2 + \norm{ \gauge_A((\ii\partial + A)f) }{L^2(\cE)}^2
	=
	\norm{f}{W_A(\cE)}^2
	\]
	for all $f \in W_A(\cE)$ since $\gauge_A$ is unitary in $L^2(\cE)$, so that $\gauge_A$ is a unitary also from $W_A(\cE)$ to
	$W^{1,2}(\cE)$. In particular, $W_A(\cE)$ is Hilbert space with respect to $\langle \cdot , \cdot \rangle_{W_A(\cE)}$ since
	$W^{1,2}(\cE)$ is a Hilbert space. Moreover, in light of \eqref{eq:boundaryUnitary}, it follows from Lemma~\ref{lem:boundary}
	that $W_A(\cE)$ belongs to $\preboundary(\cE)$ and that the restrictions of $\bdEv_\pm$ to $W_A(\cE)$ are surjective.
	
	Let $f \in W_A(\Graph)$. Using edgewise a standard Sobolev embedding and the chain rule for weakly differentiable functions
	(see, e.g., \cite[Corollary~8.11]{Bre10}), we see that $g := \gauge_A f$ belongs to $W_\loc^{1,1}(\cE)$ and satisfies
	\[
	g_\me'
	=
	(f_\me' - \ii A_\me f_\me)\exp\Bigl( -\ii\int_0^\cdot A_\me(s) \,\dd s \Bigr)
	,\quad
	\me \in \mE
	,
	\]
	so that
	\[
	(\gauge_A f)'
	=
	-\ii\gauge_A((\ii\partial + A)f)
	\in
	L^2(\cE)
	.
	\]
	Hence, $\gauge_A f \in W^{1,2}(\cE)$ and \eqref{eq:magneticDerivative} holds.
	
	Conversely, let $f = \gauge_A^*g$ with some $g \in W^{1,2}(\cE)$. By Sobolev embedding and again the chain rule, we conclude
	that $f$ belongs to $L^2(\cE) \cap W_\loc^{1,1}(\cE)$ and satisfies
	\[
	f_\me'
	=
	(g_\me' + \ii A_\me g_\me)\exp\Bigl( \ii\int_0^\cdot A_\me(s) \,\dd s \Bigr)
	,\quad
	\me \in \mE
	,
	\]
	so that
	\[
	(\ii\partial + A)f
	=
	\ii \gauge_A^* g'
	\in
	L^2(\cE)
	.
	\]
	Hence, $f \in W_A(\cE)$, which completes the proof of $\gauge_A W_A(\cE) = W^{1,2}(\cE)$ and of \eqref{eq:magneticDerivative}.
	
	Finally, the embedding of $W_A(\cE)$ in $L^2(\cE)$ is a composition of $\gauge_A$ and the embedding of $W^{1,2}(\cE)$ in
	$L^2(\cE)$. The latter is Hilbert--Schmidt if $\mE = \mE_\internal$ is finite since then each embedding of
	$W^{1,2}((0,\ell_\me))$ in $L^2((0,\ell_\me))$ is Hilbert--Schmidt by~\cite[Satz~1]{Gra68}.
\end{proof}

Let now $Y$ be a subspace of $\bdSp = \ell^2(\mE) \oplus \ell^2(\mE_\internal)$. We consider the operator $D_{A,Y}$ in
$L^2(\cE)$ defined by
\[
\Dom(D_{A,Y})
:=
\bigl\{ f \in W_A(\cE) \colon \bdEv_+(f) \in Y \bigr\}
,\quad
D_{A,Y} f
:=
(\ii\partial + A)f
\quad\text{ for }\
f \in \Dom(D_{A,Y})
.
\]
The operator $D_{A,Y}$ mimics the first (magnetic) derivative operator on intervals. It is therefore not surprising that its
adjoint corresponds to the same differential expression with appropriate boundary conditions,
cf.\ also~\cite[Theorem~3.6]{SchSeiVoi15} and \cite[Lemma~2.2.8]{Pos12}.

\begin{prop}\label{prop:derivative}
	The operator $D_{A,Y}$ is densely defined, and its adjoint $D_{A,Y}^*$ is given by
	\[
	\Dom(D_{A,Y}^*)
	=
	\bigl\{ f \in W_A(\cE) \colon \bdEv_-(f) \in Y^\perp \bigr\}
	,\quad
	D_{A,Y}^* f
	=
	(\ii\partial + A)f
	\quad\text{ for }\
	f \in \Dom(D_{A,Y}^*)
	.
	\]
	Moreover, $D_{A,Y}$ is closable with $\overline{D_{A,Y}} = D_{A,\overline{Y}}$, where $\overline{Y}$ denotes the closure of $Y$
	in $\bdSp$. In particular, $D_{A,Y}$ is closed if and only if $Y$ is closed.
\end{prop}

\begin{proof}
	We introduce the auxiliary space
	\[
	C_c^\infty(\cE)
	:=
	W^{1,2}(\cE) \cap \{ f \colon \cE \to \CC \mid f_\me \in C_c^\infty((0,\ell_\me))\ \forall \me \in \mE \}
	.
	\]	
	Since $\gauge_A^*C_c^\infty(\cE)$ is contained in $\Dom(D_{A,Y})$ by Lemma~\ref{lem:magneticSobolev} and $C_c^\infty(\cE)$ is
	dense in $L^2(\cE)$, $D_{A,Y}$ is densely defined.
	
	In order to show the representation for $D_{A,Y}^*$, we observe that integration by parts leads for $f,g \in W^{1,2}(\cE)$ to
	\[
	\langle f_\me' , g_\me \rangle_{L^2((0,\ell_\me))}
	=
	f_\me(\ell_\me)\overline{g_\me(\ell_\me)} - f_\me(0)\overline{g_\me(0)} - \langle f_\me , g_\me' \rangle_{L^2((0,\ell_\me))}
	,\quad
	\me \in \mE
	,
	\]
	with the natural understanding of $f_\me(\ell_\me) = 0 = g_\me(\ell_\me)$ if
	$\me \in \mE_\external = \mE \setminus \mE_\internal$. Given $f,g \in W_A(\cE)$, we apply the latter in light of
	Lemma~\ref{lem:magneticSobolev} to $\gauge_Af, \gauge_Ag \in W^{1,2}(\cE)$ and obtain
	\[
	\langle (\gauge_A f)' , \gauge_A g \rangle_{L^2(\cE)}
	=
	\langle \bdEv_+(\gauge_A f) , \bdEv_-(\gauge_A g) \rangle_\bdSp - \langle \gauge_A f , (\gauge_A g)' \rangle_{L^2(\cE)}
	.
	\]
	Taking into account \eqref{eq:boundaryUnitary} and \eqref{eq:magneticDerivative} and the unitarity of $\gauge_A$ and
	$\gaugeBoundary_A$, this gives
	\begin{equation}\label{eq:intByParts}
		\langle (\ii\partial + A)f , g \rangle_{L^2(\cE)}
		=
		\langle \bdEv_+(\ii f) , \bdEv_-(g) \rangle_\bdSp + \langle f , (\ii\partial + A)g \rangle_{L^2(\cE)}
		,\quad
		f,g \in W_A(\cE)
		.
	\end{equation}
	
	Set
	\[
	\cD
	:=
	\bigl\{ g \in W_A(\cE) \mid \bdEv_-(g) \in Y^\perp \bigr\}
	.
	\]
	It is then easy to see from \eqref{eq:intByParts} that $\cD \subset \Dom(D_{A,Y}^*)$ with $D_{A,Y}^*g = (\ii\partial + A)g$ for
	all $g \in \cD$. It remains to show that $\Dom(D_{A,Y}^*) \subset \cD$. To this end, let $g \in \Dom(D_{A,Y}^*)$. For every
	$\varphi \in C_c^\infty(\cE)$ we then have $\gauge_A^*\varphi \in W_A(\cE)$ with
	$D_{A,Y}\gauge_A^*\varphi = \ii\gauge_A^*\varphi'$ and, therefore,
	\[
	\langle \varphi , \gauge_A D_{A,Y}^*g \rangle_{L^2(\cE)}
	=
	\langle D_{A,Y}\gauge_A^*\varphi , g \rangle_{L^2(\cE)}
	=
	-\langle \varphi' , \ii\gauge_A g \rangle_{L^2(\cE)}
	.
	\]
	Thus, $\gauge_A g \in W^{1,2}(\cE)$ with $\ii(\gauge_A g)' = \gauge_A D_{A,Y}^*g$, that is, $g \in W_A(\cE)$ and
	$D_{A,Y}^*g = (\ii\partial + A)g$. Formula \eqref{eq:intByParts} then implies for all
	$f \in \Dom(D_{A,Y})$ that
	\begin{align*}
		\langle f , \ii\partial + A)g \rangle_{L^2(\cE)}
		&=
		\langle f , D_{A,Y}^*g \rangle_{L^2(\cE)}
		=
		\langle D_{A,Y}f , g \rangle_{L^2(\cE)}
		=
		\langle \ii\partial + A)f , g \rangle_{L^2(\cE)}\\
		&=
		\langle \bdEv_+(\ii f) , \bdEv_-(g) \rangle_\bdSp + \langle f , \ii\partial + A)g \rangle_{L^2(\cE)}
		,
	\end{align*}
	that is,
	\[
	\langle \bdEv_+(f) , \bdEv_-(g) \rangle_\bdSp
	=
	0
	,\quad
	f \in \Dom(D_{A,Y})
	.
	\]
	Since by Lemma~\ref{lem:magneticSobolev} the restriction of $\bdEv_+$ to $W_A(\cE)$ is surjective, we conclude that $\bdEv_-(g)$
	indeed belongs to $Y^\perp$, which proves the representation for $D_{A,Y}^*$. In particular, $D_{A,Y}^*$ is again densely
	defined, that is, $D_{A,Y}$ is closable.
	
	Finally, switching the roles of $\bdEv_+$ and $\bdEv_-$, the same reasoning as above shows that the closure of $D_{A,Y}$ is
	given by $\overline{D_{A,Y}} = (D_{A,Y}^*)^* = D_{A,\overline{Y}}$, where for the last equality we used that
	$(Y^\perp)^\perp = \overline{Y}$.
\end{proof}%

\begin{rem}\label{rem:boundary}
	Defining the unitary transformation $\boundarySign \colon \bdSp \to \bdSp$ by 
	\[
	\boundarySign(\alpha \oplus \beta) = (-\alpha) \oplus \beta,
	\]
	it is easy to see that $\bdEv_- = \boundarySign \circ \bdEv_+$, so that, in fact, $D_{A,Y}^* = D_{A,\boundarySign Y^\perp}$. In
	particular, $D_{A,Y}$ is self-adjoint if and only if $Y^\perp = \boundarySign Y$.
\end{rem}

Proposition~\ref{prop:derivative} allows us to introduce the magnetic Laplacian $\Delta_{A,Y}$ in $L^2(\cE)$ with
boundary conditions corresponding to a \emph{closed} subspace $Y\subset \cH$.

\begin{cor}\label{cor:Laplace}
	Suppose that $A \in L_\loc^1(\Graph)$ is real-valued and that $Y \subset \bdSp$ is a closed subspace. Then the following
	assertions hold.
	\begin{enumerate}[(a)]
		
		\item
		The quadratic form $\fa_{A,Y}$ in $L^2(\cE)$ given by
		\[
		\fa_{A,Y}[ f , g ]
		=
		\langle D_{A,Y}f , D_{A,Y}g \rangle_{L^2(\cE)}
		=
		\langle (\ii\partial + A)f , (\ii\partial + A)g \rangle_{L^2(\cE)}
		,\quad
		\Dom[\fa_{A,Y}]
		=
		\Dom(D_{A,Y})
		,
		\]
		is closed and densely defined, and the associated self-adjoint operator $-\Delta_{A,Y}$ agrees with $D_{A,Y}^*D_{A,Y}$. In
		particular, $\Delta_{A,Y}$ is given by
		\begin{align*}
			\Dom(\Delta_{A,Y})
			&=
			\bigl\{ f \in W_A(\cE) \mid
			(\ii\partial + A)f \in W_A(\cE)
			,\
			\bdEv_+(f) \in Y,\ \bdEv_-((\ii\partial+A)f) \in Y^\perp \bigr\},\\
			\Delta_{A,Y} f
			&=
			-(\ii\partial + A)^2 f
			=
			-\ii(\ii f' + Af)' - A(\ii f' + Af)
			,\quad
			f \in \Dom(\Delta_{A,Y})
			.
		\end{align*}
		
		\item
		The self-adjoint operator $D_{A,Y}D_{A,Y}^*$ agrees with the Laplacian $-\Delta_{A,\boundarySign Y^\perp}$ with
		$\boundarySign$ as in Remark~\ref{rem:boundary}, and for every $\lambda \geq 0$ the operator $D_{A,Y}$ maps the spectral
		subspace $\Ran \PP_{-\Delta_{A,Y}}(\lambda)$ into $\Ran \PP_{-\Delta_{A,\boundarySign Y^\perp}}(\lambda)$.
		
	\end{enumerate}
\end{cor}

\begin{proof}
	Part~(a) is clear by Proposition~\ref{prop:derivative}, so it suffices to prove~(b). Here, in light of
	Remark~\ref{rem:boundary}, it is easy to see that
	\[
	-\Delta_{A,\boundarySign Y^\perp}
	=
	D_{A,\boundarySign Y^\perp}^* D_{A,\boundarySign Y^\perp}
	=
	D_{A,\boundarySign(\boundarySign Y^\perp)^\perp} D_{A,\boundarySign Y^\perp}
	=
	D_{A,Y} D_{A,Y}^*
	.
	\]
	Now, consider the polar decomposition $D_{A,Y} = W\abs{D_{A,Y}}$ with a partial isometry $W$ with initial set
	$\overline{ \Ran\abs{D_{A,Y}} }$ and final set $\overline{ \Ran D_{A,Y} }$, see, e.g., \cite[Section~VI.2.7]{Kato-95}. We then
	have the identity $D_{A,Y} = W\abs{ D_{A,Y} } = \abs{ D_{A,Y}^* }W$ on $\Dom(D_{A,Y}) = \Dom(\abs{D_{A,Y}})$, and although $W$
	is only a partial isometry the latter is sufficient to conclude that
	$W \PP_{\abs{D_{A,Y}}}(\lambda) = \PP_{\abs{D_{A,Y}^*}}(\lambda) W$ for all $\lambda \geq 0$, see, e.g.,
	\cite[Lemma~3.3]{EgidiS-22}. If $f$ belongs to $\Ran \PP_{-\Delta_{A,Y}}(\lambda) = \Ran \PP_{\abs{D_{A,Y}}}(\lambda^{1/2})$, we
	therefore have $Wf \in \Ran \PP_{\abs{D_{A,Y}^*}}(\lambda^{1/2})$, so that
	$D_{A,Y} f = \abs{D_{A,Y}^*}Wf \in \Ran \PP_{\abs{D_{A,Y}^*}}(\lambda^{1/2}) = \Ran \PP_{-\Delta_{A,\boundarySign Y^\perp}}(\lambda)$.
	This completes the proof.
\end{proof}%


\newcommand{\etalchar}[1]{$^{#1}$}

\end{document}